%% file: UniFront.tex
\documentclass[11pt]{amsart} 
\usepackage{graphicx}
\usepackage{rlepsf,epstopdf, amssymb}
\usepackage[usenames]{color}
\include{header}

\newtheorem{remark}{Remark}

\title{Morse structures on open books}
\author[D. Gay]{David T. Gay}\address{Euclid Lab, Athens, GA 30606}\address{Department of Mathematics,  University of Georgia, Athens, GA 30602}\email{d.gay@euclidlab.org}
\author[J.Licata]{Joan E. Licata}\address{Mathematical Sciences Institute, The Australian National University}\email{joan.licata@anu.edu.au}
\begin{document}

\begin{abstract} 
We use parameterized Morse theory on the pages of an open book decomposition  to  efficiently encode the contact topology in terms of a  labelled graph on a disjoint union of tori (one per binding component).  This construction allows us to generalize the  notion of the front projection of a Legendrian knot from the standard contact $\mathbb{R}^3$ to arbitrary closed contact $3$-manifolds. We describe a complete set of moves on such front diagrams, extending the standard Legendrian Reidemeister moves, and we give a combinatorial formula to compute the Thurston-Bennequin number of a nullhomologous Legendrian knot from its front projection.
\end{abstract}

\maketitle

\section{Introduction}

Every contact $3$-manifold is locally contactomorphic to the standard contact $(\mathbb{R}^3, \xi_{\mathrm{std}} = \ker (dz + x\; dy))$, but this fact does not necessarily produce large charts that cover the manifold efficiently. This paper uses an open  book decomposition of a contact manifold to produce a particularly efficient collection of such contactomorphisms, together with simple combinatorial data describing how  to reconstruct the contact $3$-manifold from these charts.   This data is recorded in a \textit{Morse diagram}, and we use this perspective to define {\em front projections} for Legendrian knots and links in arbitrary contact $3$-manifolds. Our main tool is parameterized Morse theory on the pages of open books, viewed as Weinstein manifolds.  We now give more precise statements, along with a minimal set of definitions.

Let $(M,\xi)$ be an arbitrary closed contact $3$--manifold with supporting open book decomposition $(B,\pi)$, where $B$ is the binding and $\pi : M \setminus B \to S^1$ is the fibration.  Let $W = (0,\infty) \times S^1 \times S^1$ with coordinates $x \in (0,\infty)$, $y, z \in S^1$, and with contact structure $\xi_W = \ker (dz + x\; dy)$.

\begin{thm} \label{thm:contacto}
 There is a $2$-complex $\mathrm{Skel} \subset M$ with the property that, after modifying $\xi$ by an isotopy through contact structures supported by $(B,\pi)$, the complement $(M \setminus (\mathrm{Skel} \cup B), \xi)$ is contactomorphic to a disjoint union of $n=|B|$ copies of $(W,\xi_W)$. \end{thm}

We construct $\text{Skel}$ by equipping an ordinary open book $(B, \pi)$ with a certain pair $(F,V)$, where  $F$ is a real-valued function on $M$ and  $V$ a vector field on $M \setminus B$.  We call such a pair  a \textit{Morse structure}. The precise definition is given in Section~\ref{sec:mstr}, but we indicate the flavor of this object here.

\begin{defn} \label{def:morsesmale}
 An {\em efficient Morse-Smale pair} on a surface $\Sigma$ is a pair $(f,V)$ satisfying the following properties:
 \begin{itemize}
 \item $f$ is a Morse function with  one index $0$ critical point, finitely many index $1$ critical points, and no index $2$ critical points;  
 \item $V$ is gradient-like for $f$, such that ascending and descending manifolds intersect transversely in level sets. 
 \end{itemize} An {\em efficient Morse-Smale homotopy} is a $1$--parameter family $(f_t,V_t)$ such that $f_t$ is Morse for all $t$, $V_t$ is gradient-like for $f_t$ for all $t$, and $(f_t,V_t)$ is an efficient Morse-Smale pair for all but finitely many values of $t$, when handle slides occur. 
\end{defn}

Note that for an efficient Morse-Smale pair, $\Sigma$ cannot be a closed surface, and descending manifolds for index $1$ critical points all flow to the index $0$ critical point. Furthermore, in an efficient Morse-Smale homotopy,   there are no births or deaths of cancelling pairs of critical points.

A Morse structure $(F,V)$ compatible with $(M,B, \pi)$,  has two key features, as well as some technical conditions which are stated precisely in Definition~\ref{def:ms}.  First,  the restriction of a Morse structure $(F,V)$ to a single page $\Sigma_t$ is an efficient Morse-Smale pair for all but finitely many values of $t$, and  the fibration parameter $t$ determines  an efficient Morse-Smale homotopy $(f_t, V_t)$.

Second, on each page, $V_{t}$ is required to be Liouville for a symplectic form on $\Sigma_t$.   Flow along $V_t$ in the complement of $\text{Skel}$  produces the contactomorphism claimed in Theorem~\ref{thm:contacto}, and the $2$-complex $\mathrm{Skel}$ (the skeleton) referred to in Theorem~\ref{thm:contacto} is the union  over all pages of all critical points and their descending manifolds.

   \begin{figure}
\begin{center}
\scalebox{.55}{\includegraphics{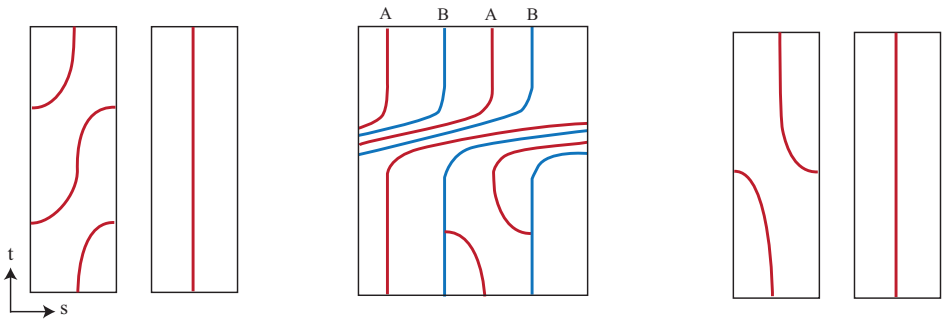}}
\caption{ Three examples of Morse diagrams.  Left: $L(2,1)$ with the universally tight contact structure. Center: An open  book with punctured torus pages with monodromy a product of a negative Dehn twist around  a non-separating curve followed by a boundary parallel positive Dehn twist.  Right: An overtwisted $S^3$. }\label{fig:ex1}
\end{center}
\end{figure}

A Morse structure similarly determines a  co-skeleton in $M$, the union of the index $1$ ascending manifolds on each page.  This $2$--complex $\text{Coskel}$ intersects the boundary of a regular neighborhood of the binding in a trivalent graph $\Gamma$, and the isotopy type of this graph on the parameterized tori determines the original contact open book $(M,\xi, B, \pi)$ as a  compactification of $\amalg^n (W,\xi_W)$,   up to diffeomorphism.    We call this collection of decorated tori a \textit{Morse diagram}.  At all but finitely many $t$ values, the Morse diagram is decorated with $2k$ paired points which trace out paired curves as $t$ varies.  At a value $t_0$ corresponding to a handle slide --- which we will call a \textit{handle slide $t$-value} --- the Morse diagram has $2k-2$ ordinary points and $2$ double points.  The double points are endpoints of a single edge with a discontinuity at $t_0$ value; as $t\rightarrow t_0^+$, the edge approach one edge in a different pair from the left (respectively, right), and as $t\rightarrow t_0^-$, the edge approaches the paired edge from the right (left). See the central picture in Figure~\ref{fig:ex1} for an example of a handle slide on a Morse diagram.

 We will see in Section~\ref{sec:contacto} that up to diffeomorphism, the Morse diagram coming from a Morse open book  uniquely determines $(M,\xi, B, \pi)$  as a compactification of $\amalg^n (W,\xi_W)$.  Furthermore,  any Morse diagram is the compactification data for some $(M,\xi,B,\pi)$.

In addition to combinatorially encoding the contact manifold, the Morse diagram functions as a target for defining the front projection of a Legendrian link in an open book with a Morse structure.

\begin{defn}\label{def:front}
 A {\em front} on a Morse diagram $\Gamma \subset \amalg^n S^1 \times S^1$ is a collection of arcs and closed curves $\mathcal{F}$ immersed, with semicubical cusps, in $\amalg^n S^1 \times S^1$, satisfying the following properties:
 \begin{enumerate}
  \item The slopes at all interior points on $\mathcal{F}$ are negative (using coordinates $(s,t)$ on $S^1 \times S^1$ and measuring slope as $dt/ds$).
  \item The endpoints of arcs of $\mathcal{F}$ lie on the interiors of curves of $\Gamma$ and have slope $0$.
  \item Suppose that $e$ and $e'$ are two edges of $\Gamma$ with the same label. For every arc of $\mathcal{F}$ ending on $e$ at height $t$, approaching $e$ from the left (respectively, right), there is an arc of $\mathcal{F}$ ending on $e'$ at height $t$, approaching $e'$ from the right (left). 
 \end{enumerate}
\end{defn}
Figure~\ref{fig:ex2} shows some fronts on the Morse diagrams of Figure~\ref{fig:ex1}.

   \begin{figure}
\begin{center}
\scalebox{.55}{\includegraphics{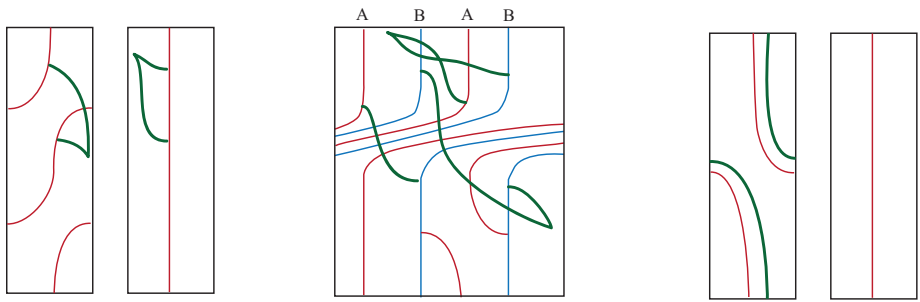}}
\caption{ The bold curves are front projections of Legendrian knots.  The right-hand example is the boundary of an overtwisted disc in $S^3$, while the other two examples represent non-trivial homology classes in the closed $3$-manifold.}\label{fig:ex2}
\end{center}
\end{figure}

\begin{thm} \label{thm:fronts}
 Let $\Lambda$ be a Legendrian link in $(M,\xi)$ that is disjoint from the binding and  transverse to $\mathrm{Skel}$. Then the image of $\Lambda$ under the flow by $\pm V$ to $\amalg^n  S^1 \times S^1$  is a front on the Morse diagram. Furthermore, any front on this Morse diagram is the image of such a Legendrian $\Lambda$, and any two Legendrians with the same front are equal.
\end{thm}

In Section~\ref{sec:isotopymoves} we describe a list of moves on fronts, which we call {\em isotopy moves}, and we show the following:

\begin{thm}\label{thm:frontmoves}
 In the setting of Theorem~\ref{thm:fronts}, two Legendrian links in $(M,\xi)$ are Legendrian isotopic if and only if their fronts are related by a sequence of isotopy moves.
\end{thm}

In Section~\ref{sec:tb}, we show how to detect whether or not a front is the front projection of a nullhomologous Legendrian knot; if it is, we show how to construct Seifert surfaces for such fronts, define the {\em total writhe} $W$ of a front $\mathcal{F}(\Lambda)$ and then prove the following:

\begin{thm}\label{thm:tb}
  If $\Lambda$ is a nullhomologous Legendrian knot, then the Thurston-Bennequin number $\mathrm{tb}(\Lambda)$ is equal to $W(\mathcal{F}(\Lambda)) - \frac{1}{2} |\mathrm{cusps}|$.
\end{thm}

\subsection{Acknowledgements}
This paper grew from discussions at the Banff International Research Station in 2013 which were continued there in 2014, and the authors are grateful for the hospitality and support of BIRS. The second author  appreciated the chance to visit the University of Georgia for continued work, and the first author is  partially supported by NSF grant DMS-1207721.

\section{Background and notation}\label{sec:background}
All our manifolds are oriented and connected unless otherwise stated, and contact structures are co-oriented.

As a matter of convenience we will sometimes identify $S^1$ with $\mathbb{R}/\mathbb{Z}$ and sometimes with $\mathbb{R}/2\pi\mathbb{Z}$; which we are using should be clear from context. We will use a Roman letter variable name, such as $t$, in the $\mathbb{R}/\mathbb{Z}$ case, and a Greek letter variable name, such as $\theta$, in the $\mathbb{R}/2\pi\mathbb{Z}$ case. We will use $(\rho,\mu,\lambda)$ for polar coordinates on the solid torus $\mathbb{R}^2 \times S^1$, with $\mu, \lambda \in \mathbb{R}/2\pi\mathbb{Z}$, and $(\rho,\mu)$ being standard polar coordinates on $\mathbb{R}^2$, i.e., $\mu$ for ``meridian'' and $\lambda$ for ``longitude''.

Here and throughout, suppose that $(B, \pi)$ is an open book decomposition of a closed connected oriented $3$-manifold $M$. That is, $B$ is an oriented link in $M$ and $\pi:M\setminus B\rightarrow S^1$ is a fibration with the property that for all $t\in S^1$, the closure of $\pi^{-1}(t)$ is a Seifert surface for $B$.  Each compact connected surface $\Sigma_t = B\cup \pi^{-1}(t)$ is a \textit{page} of the open book.

The utility of open books for the study of contact geometry comes from the following notion of compatibility between open book decompositions and contact structures on a fixed manifold: 

\begin{defn}\label{def:ob} A contact form $\alpha$ on $M$ is \textit{compatible} with $(B,\pi)$ if
\begin{itemize}
\item for all $t$, $d\alpha|_{\Sigma_t}$ is a symplectic form; and
\item $\alpha |_B > 0$, where $B$ is oriented as the boundary of any $\Sigma_t$, and $\Sigma_t$ is oriented by $d\alpha$.
\end{itemize}
A contact structure $\xi$ on $M$ is \textit{supported} by $(B,\pi)$ if there exists a contact form for $\xi$ which is compatible with $(B,\pi)$.
\end{defn}

When we refer to a $4$--tuple $(M,\xi,B,\pi)$, we always imply that $\xi$ is supported by $(B,\pi)$. Diffeomorphisms between such $4$-tuples are contactomorphisms respecting the open book structure.

\begin{thm}[Thurston-Winkelnkemper~\cite{TW}, Giroux~\cite{Gi}] \label{thm:obdcs}
Each open book supports a unique isotopy class of contact structures.
\end{thm}

\begin{defn}
Given a surface $\Sigma$ with boundary, the \textit{ mapping class group}  $\text{Mod}(\Sigma)$ is the group of orientation preserving self-diffeomorphisms of $\Sigma$ which are equal to the identity on $\partial \Sigma$, modulo isotopies fixing $\partial \Sigma$ pointwise.
\end{defn}

\begin{defn}
 A vector field $X$ on $M \setminus B$ is a {\em monodromy vector field} if $d\pi(X)=1$ and if, for each component of $B$, there are solid torus coordinates $(\rho,\mu,\lambda) \in D^2 \times S^1$ with respect to which $\pi = \mu$ and $X = \partial_\mu$.
\end{defn}

Note that monodromy vector fields exist (use partitions of unity), and that a monodromy vector field allows one to read off the monodromy of an open book as the return map on a page, and thus to identify $(M,B,\pi)$ with the {\em abstract open book} $M(\Sigma,\phi) = \Sigma \times [0,1]/\!\sim$, where $(p,1) \sim (\phi(p),0)$ for all $p \in \Sigma$ and $(p,s) \sim (p,t)$ for all $p \in \partial \Sigma$ and for all $s,t \in [0,1]$. In particular, this also allows us to identify every page $\Sigma_t$ with a fixed model page $\Sigma = \Sigma_0$. Furthermore, different monodromy vector fields produce isotopic return maps, so that the monodromy is well defined as an element of the mapping class group $\mathrm{Mod}(\Sigma)$.

\section{Morse structures and Morse diagrams}\label{sec:mstr}
In preparation for constructing Morse structures, we briefly depart from the world of contact geometry and consider open books as  topological    objects. 
\begin{defn} \label{def:ms}
 Let $M$ be a closed, connected, oriented $3$--manifold with an open book $(B,\pi)$. A {\em Morse structure} on $(M,B,\pi)$ is a pair $(F,V)$, where $F: M \to (-\infty,0]$ is a smooth function and $V$ is a smooth vector field on $M$ satisfying the following properties:
 \begin{enumerate}
  \item $F^{-1}(0)=B$.
  \item\label{note1} There is a solid torus neighborhood of each component of $B$, with coordinates $(\rho,\mu,\lambda)$, on which $B = \{\rho=0\}$, $\pi = \mu$, $F = -\rho^2$ and $V = - (\rho/2) \partial_\rho$. 
  \item\label{note2} On the interior of each page $\Sigma_t \setminus B$, the function $f_t: = F|_{\Sigma_t \setminus B}$ is Morse.  
   \item $V$ is tangent to each page $\Sigma_t$, and $V_t := V|_{\Sigma_t}$ is gradient--like for $f_t$.
  \item Using a monodromy vector field $X$ to identify each page $\Sigma_t$ with a fixed page $\Sigma$, the family $(f_t,V_t)$ on $\Sigma \setminus \partial \Sigma$ is an efficient Morse-Smale homotopy. 
 \end{enumerate}
 A {\em Morse open book} on $M$ is a $4$--tuple $(B,\pi,F,V)$ such that $(F,V)$ is a Morse structure on $(M,B,\pi)$.
\end{defn}

\begin{rem}
The factor of $1/2$ in (\ref{note1}) is not significant here, but becomes convenient when we bring the contact geometry back  to the story in Section~\ref{sec:morsecontact}.
\end{rem}

\begin{prop} \label{prop:obdms}
 Every open book decomposition has a Morse structure.
\end{prop}

\begin{proof}
 We need only   show that, given a surface $\Sigma$ with the correct genus and number of boundary components and a mapping class $\Phi \in \text{Mod}(\Sigma)$, there exists an efficient Morse-Smale homotopy $(f_t,V_t)$ on $\Sigma$ with $\phi^*(f_0,V_0) = (f_1,V_1)$ for some representative $\phi$ of $\Phi$. Once we do this, the rest follows by constructing the mapping torus $M(\Sigma, \phi)$ and gluing in a solid torus neighborhood of the binding. (The initial $\Sigma$ ends up being the complement in the page $\Sigma_0$ of a collar neighborhood of $\partial \Sigma_0$.) 
 
First choose an initial Morse function $f_0$ and a representative $\phi$, and let $f_1=\phi^* f_0$. In~\cite{GK}, it is shown how to eliminate extraneous minima and maxima (index $0$ and $2$ critical points in this case) in a generic path connecting Morse functions; apply this to get $f_t$. Then standard Cerf theory gives us $V_t$, interpolating from some chosen $V_0$ to $V_1 = \phi^* V_0$.
\end{proof}

There are two natural subcomplexes in $M$ associated to a Morse open book $(B,\pi,F,V)$:
\begin{defn} \label{def:skel}
 The {\em skeleton}, denoted $\mathrm{Skel}$, is the union over all pages of the descending manifolds of the index $1$ critical points of $(f_t,V_t)$, together with the index $0$ critical point. The {\em co-skeleton}, denoted $\mathrm{Coskel}$, is the union over all pages of the ascending manifolds of the index $1$ critical points.
\end{defn}

\begin{remark}
 The intersection of $\mathrm{Skel}$ and $\mathrm{Coskel}$ is a $1$--complex consisting of the index $1$ critical points in all pages together with a flow line between two index $1$ critical points at each handle slide $t$--value.
\end{remark}

Given a Morse open book $(B,\pi,F,V)$ on $M$, fix coordinates $(\rho,\mu,\lambda)$ on a solid torus neighborhood of each component of $B$ as in Definition~\ref{def:ms}.  Let $n = |B|$, and embed $\amalg^n S^1 \times S^1$ in $M$ as $\{\rho^2 = \epsilon\} = F^{-1}(-\epsilon)$, for suitably small $\epsilon$, mapping coordinates $(s,t)$ on $S^1 \times S^1$ to $(\lambda,\mu)$ on $\{\rho^2=\epsilon\}$. Denote these embedded tori as $\amalg_{i=1}^{n} \mathcal{T}_i$
.
\begin{defn}\label{def:md}
 The \textit{Morse diagram associated to} $(B,\pi,F,V)$ is the collection of decorated tori  \[ (\amalg \mathcal{T}_i, \text{Coskel} \cap \amalg \mathcal{T}_i),\] together with a pairing of curves on the tori corresponding to the same index $1$ critical points.
\end{defn}

 We may characterize the kinds of decorated tori that can occur as Morse diagrams.  We do so here, and when necessary, we may distinguish between the terms  \textit{embedded Morse diagram}, which denotes  tori in the contact manifold whose decoration comes from the intersection with the co-skeleton as described in Definition~\ref{def:md}, and  \textit{abstract Morse diagram}, which  denotes any collection of decorated tori satisfying the description below.

\begin{defn} \label{def:abstractmd}
 An  \em{abstract Morse diagram} is a collection of tori $\amalg^n S^1 \times S^1$ with a finite trivalent graph $\Gamma$  such that
 \begin{itemize}
 \item the edges of $\Gamma$ are monotonic with respect to projection to the second $S^1$ factor; 
 \item for each fixed value $c$ of the second factor, there is a pairing on curves intersecting $\amalg^n S^1\times c$, and the pairing is constant away from vertices;
 \item surgery on $\amalg^n S^1\times c$ with attaching spheres given by paired points on the curves yields a single $S^1$;
 \item trivalent points occur in pairs on the same slice $\amalg^n S^1\times c$.  As $t\rightarrow c^-$, a curve labelled $x$ approaches a curve labelled $y$ from the left (respectively, right), while as $t\rightarrow c^+$, a curve labelled $x$ approaches the other curve labelled $y$ from the right (left).
  \end{itemize}
 
  An isotopy of Morse diagrams is a smooth $1$--parameter family of such graphs with pairings.
\end{defn}
Figure~\ref{fig:ex1} shows a few examples; the pairing data is represented by edge labelings.

\begin{prop} \label{prop:mdmob}
 Every abstract Morse diagram is a Morse diagram associated to some open book. If two Morse open books have isotopic Morse diagrams, then the $3$-manifolds are diffeomorphic via a diffeomorphism respecting the open books.
\end{prop}

Note that we do not claim that the diffeomorphism respects the Morse structures, although if  needed, one could describe an appropriate equivalence relation on Morse structures to make this true. The main issue is that the Morse diagram does not record the relative ordering of the critical values of $f_t$.

\begin{proof}
It suffices to give a construction starting from a Morse diagram which makes it clear that the Morse diagram determines the diffeomorphism type of the page and the mapping class of the monodromy (up to conjugation).

We start by recalling the standard construction of a handle structure on a surface built by attaching $k$ $2$-dimensional $1$-handles to a disc.  Glue each handle $[-1,1]\times [-1,1]$ to the disc along $\{\pm 1\} \times [-1,1]$; the co-core of the handle is the arc $\{0\}\times [-1,1]$, and we note that the co-core intersects the boundary of the new surface in a pair of points.  Similarly, concatenating the core curves $[-1,1] \times \{0\}$ with rays to the center of the disc produces a wedge of circles onto which  the new surface deformation retracts, and we refer to each such loop as a core. 

On the Morse diagram,  the level $\amalg^n S^1\times \{0\}$ intersects the trace curves in $2k$ paired marked points. Construct a model surface $\Sigma$ with handle decomposition as above with the property that the natural handle structure induces the same pairing of marked points around the boundary.  Remove an open neighborhood of the center of the original disc and denote the resulting surface by $\widetilde{\Sigma}$.   

The union of the cores and co-cores of $\Sigma$ form a collection $\mathcal{H}$ of properly embedded arcs in $\widetilde{\Sigma}$.  Before continuing, we note a few useful properties of this set.  These curves  cut $ \widetilde{\Sigma}$ into a collection of discs.  They are pairwise non-isotopic, intersect minimally, and have the property that for any three arcs $\gamma_i$, $\gamma_j$, and $\gamma_k$, at least one of the pairwise intersections is empty.  According to Proposition 2.8 in \cite{FM}, the action of the mapping class group of  $\widetilde{\Sigma} $ acts faithfully on the graph formed by any set of curves satisfying these properties.

Now consider the product $\widetilde{\Sigma}\times [0,1]$ and let $\mathcal{H}_0 = \mathcal{H} \times \{0\}$, on $\widetilde{\Sigma}\times \{0\}$. By construction, $\mathcal{H}_0 \cap  \partial \widetilde{\Sigma}\times \{0\}$ agrees with the decoration on the Morse diagram at $t=0$. Perform the necessary isotopies and handle slides on the co-cores so that, for each $t$, $\mathcal{H}_t$ comes from a handle structure on $\widetilde{\Sigma}\times \{t\}$ with the property that the order of the marked points on $\partial \widetilde{\Sigma}\times \{t\}$ agrees with the order of the marked points on the the corresponding $t$-slice of the Morse diagram.   The cores are determined up to isotopy by the co-cores, and the co-cores are  determined by the Morse diagram, so this process  determines $\mathcal{H}_t$ up to isotopy for $t\in [0,1]$.  In order to form a mapping torus from $\widetilde{\Sigma}\times [0,1]$ which identifies $\mathcal{H}_0$ and $\mathcal{H}_1$, we require a diffeomorphism $\phi$ of $\widetilde{\Sigma}$ with the property that $\phi_*(\mathcal{H}_0)$ is isotopic to $\phi_*(\mathcal{H}_1)$.   The proposition noted above implies that  the mapping class of such a $\phi$ is uniquely determined; extended to $\Sigma$, this is the monodromy of the open book. 

The construction above allowed for choice in the original handle structure, but the surfaces associated to any  two such choices are related by a diffeomorphism identifying the initial handle structures.  Making another such choice, the construction above recovers the conjugate of $\phi$ by this diffeomorphism, which determines a diffeomorphic open book. 

We have shown that the $3$--manifold with its open book is determined by the Morse diagram. To see that the Morse diagram actually comes from a Morse structure, realize the $1$--parameter family of handle decompositions by pairs $(f_t,V_t)$. This involves further choices, for which we do not claim any uniqueness.
\end{proof}

\section{Contact structures, Morse structures and Morse diagrams}\label{sec:morsecontact}

Since a Morse diagram determines an open book, and an open book determines a unique isotopy class of contact structures, we immediately have the fact that Morse diagrams describe contact $3$-manifolds. However, to achieve the contactomorphism of Theorem~\ref{thm:contacto} and the consequent generalized front projections, we need a more rigid relationship between Morse structures and contact topology.

\begin{defn} \label{def:mscs}
 Suppose that $(F,V)$ is a Morse structure on $(M,B,\pi)$, and that $\xi$ is a contact structure on $M$ supported by $(B,\pi)$. We say that $(F,V)$ is {\em compatible} with $\xi$ if there is a contact form $\alpha$ for $\xi$ on $M \setminus B$  satisfying the following conditions:
 \begin{enumerate}
  \item On the interior of each page $\Sigma_t \setminus B$, $\omega_t: = (d\alpha)|_{\Sigma_t \setminus B}$ is symplectic and $V_t$ is Liouville for $\omega_t$.
  \item There is a monodromy vector field $X$ such that $\alpha(X) = 1$.
  \item In the given local solid torus coordinates $(\rho,\mu,\lambda)$ near each component of $B$,     $\alpha = (1/\rho^2) d\lambda + d\mu$.
 \end{enumerate}
\end{defn}

Note that the form  $\alpha$ will not extend across $B$.

\begin{lem} \label{lem:mscfcs}
 Given $(M,B,\pi)$ and a contact form $\alpha$ on $M \setminus B$, suppose that, in local coordinates $(\rho,\mu,\lambda)$ near each component of $B$, we have $B = \{\rho=0\}$, $\pi=\mu$, and $\alpha = (1/\rho^2) d\lambda + d\mu$. Then $\ker \alpha$ extends across $B$ and, if $d\alpha$ is positive on the interior of each page, is a contact structure supported by $(B,\pi)$
\end{lem}

\begin{proof}
The $1$-form $\alpha = (1/\rho^2) d\lambda + d\mu$
has the same kernel as $d\lambda + \rho^2 d\mu$ and hence, the associated contact structure extends over the core $\{0\} \times S^1 \subset D^2 \times S^1$. In fact, both have the same kernel as $(1/(1+\rho^2))d\lambda + (\rho^2/(1+\rho^2))d\mu$, which also extends over the core and has its Reeb vector field transverse to pages. Furthermore, by choosing functions $(f(\rho),g(\rho))$ that interpolate appropriately between $(1/(1+\rho^2),\rho^2/(1+\rho^2))$ and $(1/\rho^2,1)$, one can produce a single contact form that agrees with $\alpha$ outside a solid torus neighborhood of the binding, extends across the binding, and is supported by $(B,\pi)$.
\end{proof}

\begin{prop} \label{prop:contmob}
 Given any contact $3$--manifold $(M,\xi)$ with supporting open book $(B,\pi)$, there is a Morse structure on $(M,B,\pi)$ compatible with a contact structure $\xi^*$ which is isotopic to $\xi$ through contact structures supported by $(B,\pi)$.
\end{prop}

\begin{proof}

Rather than constructing $(F,V)$ from $\xi$, we construct a $3$--manifold with open book $(M',B',\pi')$ diffeomorphic to $(M,B,\pi)$, starting from the page and monodromy of $(M,B,\pi)$, and along the way we construct $(F',V')$ and $\alpha'$ on $M' \setminus B'$ satisfying the conditions in Definition~\ref{def:mscs}. In the construction it will be clear that $F'$ and $V'$ extend across $B'$, and by Lemma~\ref{lem:mscfcs}, $\alpha'$ extends to a contact {\em structure} $\xi'$ compatible with $(B',\pi')$. Pull all this data back to $(M,B,\pi)$ by the diffeomorphism and we get a contact structure $\xi^*$ on $M$ supported by $(B,\pi)$ and a Morse structure $(F,V)$ for $(B,\pi)$ compatible with $\xi^*$. Finally, by Theorem~\ref{thm:obdcs}, $\xi^*$ is isotopic to $\xi$ through contact structures supported by $(B,\pi)$.

Having explained the structure of the proof, we will now give the construction but we will drop the ``primes'' from our notation to simplify the exposition.

Recall that the {\em page} $\Sigma_0$ of the given $(M,B,\pi)$ is the compact surface $\pi^{-1}(0) \cup B$. Let $\Sigma = \pi^{-1}(0)$, the noncompact interior of $\Sigma_0$. Let $(-\epsilon,0] \times B$ parametrize a collar neighborhood of $B = \partial \Sigma_0$, so that $\partial \Sigma_0 = \{0\} \times B$. Let $E = (-\epsilon,0) \times B$, the union of the ``ends'' of $\Sigma$, and reparameterize $E$ as $(-\epsilon,\infty) \times B$ using an arbitrary orientation preserving diffeomorphism $(-\epsilon,0) \to (-\epsilon,\infty)$. Use coordinates $(r,s)$ on $E$, where $r \in (-\epsilon,\infty)$ and $s$ is an $\mathbb{R}/\mathbb{Z}$ coordinate on each component of $B$.

Choose a Morse function $f$ on $\Sigma$ which equals $r$ on $E$, with a single index $0$ critical point $p \in \Sigma \setminus E$ and with no index $2$ critical points.
\begin{lem} 
There exists a $1$--form $\delta$ and a vector field $V$ on $\Sigma$ satisfying the following conditions:
\begin{enumerate}
 \item On $E$, $\delta = (1+r) ds$ and $V=(1+r)\partial_r$.
 \item On all of $\Sigma$, $d\delta >0$.
 \item $(f,V)$ is an efficient Morse-Smale pair.
 \item $V$ is Liouville for $d\delta$.
\end{enumerate}
\end{lem}

The proof of the lemma follows from a standard Weinstein handle construction~\cite{Weinstein}, and we leave the details to the reader. We now use the language of Weinstein cobordisms, following Cieliebak and Eliashberg~\cite{CY} to extend this structure to the rest of the mapping torus.

A {\em Weinstein cobordism} is a $4$--tuple $(W,\omega,V,f)$ where $W$ is a compact cobordism from $\partial_- W$ to $\partial_+ W$, $\omega$ is a symplectic form on $W$, $f: W \to [a,b] \subset \rr$ is a Morse function with $\partial_+ W = f^{-1}(b)$ and $\partial_- W = f^{-1}(a)$, and $V$ is a Liouville vector field for $\omega$ which is gradient-like for $f$. In particular, $V$ points in along $\partial_- W$ and out along $\partial_+ W$.  A {\em Weinstein homotopy} is a smooth family $(W, \omega_t,V_t,f_t)$, with $t \in [0,1]$ which is a Weinstein cobordism except at finitely many  $t$ where birth-death singularities may occur for $f_t$, and hence for $V_t$ as well.

Returning to the noncompact page $\Sigma$ with Morse function $f$, let  $W=f^{-1}([f(p)+\epsilon,0])$ for sufficiently small $\epsilon>0$, so that there are no critical values in $(f(p),f(p)+\epsilon]$ and thus $\partial_-W=\partial \nu(p)$ for some disk neighborhood $\nu(p)$ of the unique index $0$ critical point $p$. We claim that there exists a diffeomorphism $\phi : \Sigma_0 \to \Sigma_0$ representing the monodromy of $(M,B,\pi)$ which is the identity on $\Sigma_0 \setminus W$, such that, restricting $\phi$ to $W$, there exists a Weinstein homotopy  $(W, d\delta, V_t, f_t)$  satisfying the following conditions:
\begin{enumerate}
  \item $(f_0,V_0) = (f,V)$
  \item $D\phi(V_1) = V_0$
  \item $f_0 \circ \phi = f_1$
  \item $(f_t,V_t)=(f_0,V_0)$ on a neighborhood of $\partial W$.
  \item $(f_t,V_t)$ is an efficient Morse-Smale homotopy.
\end{enumerate}
Lemma~\ref{lem:WeinsteinHomotopy} below asserts the existence of such a homotopy in the case that $\phi$ is a Dehn twist along a nonseparating or boundary parallel simple closed curve in $W$, and its proof occupies Section~\ref{sec:homotopy} below. The claim then follows from the fact that the mapping class group of a surface is generated by Dehn twists along such curves~\cite{Lick}. 

Identifying the homotopy parameter with the $t$ parameter in $\Sigma\times[0,1]$, any such Weinstein homotopy defines a smooth function $F$ and a gradient-like vector field $V$ on the mapping torus $\Sigma \times [0,1]/(p,1) \sim (\phi(p),0)$.    Furthermore, we define a $1$-form on each page $ \Sigma\times\{t\}$ by taking the contraction of $d\delta$ with $V_t$:
\[ \delta_t := \iota_{V_t}d\delta.\] 
By construction, this defines a one-form $\Lambda$ on the mapping torus (via $\Lambda(\partial_t)=0$ and $\Lambda|_{\Sigma_t} = \delta_t$), and for sufficiently large $K$,  $\alpha=\Lambda+Kdt$ is contact.  The standard construction of a contact form on a mapping torus uses a linear interpolation between $\delta_0$ and $\phi^*\delta_0$ where we have used $\Lambda$, but we see that this construction agrees with our $\alpha$ on  the end(s) $E$ (and on a neighborhood of the transverse knot traced out by the index $0$ critical point $p$).  We extend $\alpha$ across the binding of the open book as described next; this extension is a slight modification of the extension described in~\cite{OS}.

In order to construct a closed $3$-manifold $M(\Sigma_0,\phi)$, and thus transport this data to $(M,B,\pi)$, use coordinates  $(r,s,t)$ on the ends $E \times [0,1]/\!\sim$ of the mapping torus and glue in solid tori $D^2\times S^1$, with coordinates  $(\rho, \mu, \lambda)$, with $0 \leq \rho^2 \leq K/(1-\epsilon)$, via the map
\[ \psi(\rho, \mu, \lambda)=((K/\rho^2)-1, \lambda, \mu).\]
Note that $\psi^* \alpha = (K/\rho^2) d\lambda + K d\mu$; globally rescaling by $1/K$ gets the desired local model $(1/\rho^2) d\lambda + d\mu$. A simple calculation shows that $V$ has the correct local model, and an appropriate reparameterization of $(-\epsilon,\infty)$ as $(-\epsilon,0)$ is all that is needed to modify $F$ to have the correct local model. The monodromy vector field is $\partial_t$.
\end{proof}

\subsection{Homotopy existence}\label{sec:homotopy}

This section gives a technical construction of the homotopy whose existence was  claimed above.  Instead of considering an arbitrary monodromy map, we restrict to the case  where $\phi$ is a single Dehn twist $\tau_C$ about a curve $C$. We assume  $C$ is either non-separating or boundary parallel; this suffices to establish the general case, as the relative mapping class group of a surface is generated by such Dehn twists \cite{Lick}. 

To use the lemma presented below in the proof above, one needs to choose appropriate minor reparameterizations in the collar neighborhood direction, which is mapped via the Morse function to $[0,1]$ here, but may need to be mapped to some $[-a,0]$, for example. Also, here the Liouville vector field is presented in the form $\partial_\zeta$ instead of, for example, $(1+r) \partial_r$. This leads to an exponential $e^{k \zeta}$ appearing in the symplectic form. A further standard reparameterization takes care of this. 

\begin{lem} \label{lem:WeinsteinHomotopy}
 Suppose we are given the following data:
\begin{enumerate}
 \item A compact, connected, oriented $2$-dimensional cobordism $W$ from $\partial_0 W$ to $\partial_1 W$, with each $\partial_i W$ a compact, oriented, nonempty, possibly disconnected, $1$--manifold.
 \item A fixed parametrization of a collar neighborhood $\kappa$ of $\partial W$ as $\kappa = ([0,\epsilon] \times \partial_0 W) \amalg ([1-\epsilon,1] \times \partial_1 W)$.
 \item A Morse function $f: W \to [0,1]$ with $f^{-1}(i) = \partial_i W$ for $i=0,1$, and with only critical points of index $1$, with distinct critical points having distinct critical values, with $f|_{\kappa}$ being projection onto the first factor $[0,\epsilon] \amalg [1-\epsilon] \subset [0,1]$.
 \item An area form $\beta$ on $W$ with $\beta|_{\kappa} = e^{k \zeta} d\zeta \wedge d\theta$, where $\zeta$ is the $[0,1]$ coordinate on $\kappa$, $\theta \in [0,2\pi]$ is a fixed oriented angular coordinate on each component of $\partial W$, and $k$ is some positive constant. (We may need $k > 1$ if there are many more components of $\partial_0 W$ than $\partial_1 W$, for example.)
 \item A vector field $V$ on $W$ which is gradient-like for $f$ and is Liouville for $\beta$ (i.e. $d \gamma = \beta$ where $\gamma = \imath_V \beta$), with $V|_{\kappa} = \partial_\zeta$. Here we also assume that the ascending and descending manifolds, with respect to $V$, of distinct critical points of $f$, are disjoint.
 \item A  simple closed curve $C \subset W \setminus \kappa$ which is either non-separating or boundary parallel.
\end{enumerate}
 Then there exists a $1$--parameter family of pairs $(f_t,V_t)$, with $t \in [0,1]$, satisfying the following properties:
\begin{enumerate}
 \item $f_0 = f$ and $V_0=V$.
 \item $f_t = f_0$ and $V_t = V_0$ for $t \in [0,\epsilon]$, while $f_t =f_1$ and $V_t = V_1$ for $t \in [1-\epsilon,1]$.
 \item For each $t \in [0,1]$, $f_t$ is Morse, has critical points only of index $1$, and is projection on $[0,\epsilon] \amalg [1-\epsilon]$ on the collar neighborhood $\kappa$.
 \item For each $t \in [0,1]$, $V_t$ is gradient-like for $f_t$, Liouville for $\beta$, and equals $\partial_\zeta$ on $\kappa$.
 \item For all but finitely many values of $t$, distinct critical points of $f$ have distinct critical values, and at each of those finitely many values of $t$, precisely two critical values of distinct critical points cross transversely.
 \item For all but finitely many values of $t$ (distinct from the special values of $t$ in the preceding item), the ascending and descending manifolds, with respect to $V_t$, of distinct critical points of $f_t$, are disjoint. At those finitely many values, handle slides occur.
 \item For some curve $C'$ isotopic to $C$, and some Dehn twist $\tau_{C'}$ about $C'$, $(\beta,f_1,V_1) = \tau_{C'}^* (\beta,f_0,V_0)$.
\end{enumerate}
\end{lem}

\begin{rem}
 We emphasize in the above that Dehn twists about curves are only well-defined up to isotopy, and we do not achieve an exact given Dehn twist, but only a carefully constructed representative of its isotopy class. 
\end{rem}

\begin{proof}[Proof of Lemma~\ref{lem:WeinsteinHomotopy}]
 We will construct the homotopy in stages. From time $t=0$ to $t=1/3$, we construct a homotopy so that a curve $C'$ isotopic to $C$ is contained in the level set $f_{1/3}^{-1}(1/2)$, with $(\beta, f_{1/3},V_{1/3})$ having a nice form in a neighborhood of $C'$. From time $t=1/3$ to $t=2/3$ we implement the Dehn twist about $C'$, so that $(\beta,f_{2/3}, V_{2/3})=\tau_{C'}^* (\beta,f_{1/3},V_{1/3})$. Then from time $t=2/3$ to $t=1$ we run the original $0 \leq t \leq 1/3$ homotopy backwards, but pulled back via the Dehn twist $\tau_{C'}$, giving us a homotopy from $(\beta,f_{2/3},V_{2/3}) = \tau_{C'}^*(\beta,f_{1/3},V_{1/3})$ to $(\beta,f_1,V_1) = \tau_{C'}^*(\beta,f_0,V_0)$.

To construct the $t\in [0,\frac{1}{3}]$ stage of the homotopy,  we first consider only  the Morse theory. With $C$ as above, we can construct a Morse function $g$ on $W$ such that $C \subset g^{-1}(1/2)$, with only critical points of index $1$. Then there exists a generic homotopy from $f=f_0$ to $g$, which will be Morse for all but finitely many times, at which times births or deaths occur. However, it is standard that, because we have no index $0$ or $2$ critical points at the beginning and end of this homotopy, we can eliminate index $0$ and $2$ critical points at all intermediate times. (See, for example, Theorems~3 and~4 in~\cite{GK}.) In this low-dimensional case, we are left with only index $1$ critical points. In particular, there are in fact no births or deaths and the homotopy is Morse at all times.

 Now we consider $V=V_0$, together with $f$, as inducing a handle decomposition of $W$. The handle decomposition is not enough to recover the isotopy classes of the level sets of $f$, but these can be recovered from the handle decomposition together with the ordering of the critical points according to height (i.e., the value of $f$). In other words, any two Morse functions with gradient-like vector fields inducing isotopic ordered handle decompositions have isotopic level sets. A gradient-like vector field for $g$ also induces  an ordered handle decomposition, and standard Cerf theory then implies  that there is a sequence of handle slides and reorderings of handles which transforms  the initial ordered handle decomposition for $f$ to the ordered handle decomposition for $g$. 

 Now we will produce a homotopy $(f_t,V_t)$, for $0 \leq t \leq 1/3$ with the property that the sequence of handle slides and handle reorderings is combinatorially the same as the  sequence  produced by Cerf theory to transform $f$ to $g$.  In order to show that such a homotopy exists, we appeal to the discussion of Weinstein homotopies developed in \cite{CY}.  The existence of a homotopy reordering critical points follows immediately from Lemma 12.20.  
 
 The existence of a homotopy realizing a handle slide follows from Lemma 12.18, as we briefly explain.  Consider two index-one critical points $X$ and $Y$ with $f(X)=a < f(Y)=b$, and suppose that in the desired handle slide, the handle corresponding to $Y$ slides over the handle corresponding to $X$. This means that, in any level set between $a$ and $b$, the descending manifold for $Y$ slides across the ascending manifold for $X$. Let $W'=f^{-1}[a+\epsilon, b-\epsilon]$, seen as a Weinstein cobordism by restricting the auxiliary data from $W$. Locally, the descending  manifold of $Y$  intersects $\partial W'$ in a pair of points $p_+ \in f^{-1}(a +\epsilon)$ and $p_-\in f^{-1}(b-\epsilon)$.  The points $p_\pm$ are isotropic $0$-manifolds in the contact $1$-manifold $\partial W'$, so any smooth isotopy of these points is trivially a Legendrian isotopy.  Lemma 12.18 states that any Legendrian isotopy of $p_-$ can be realized by a Weinstein homotopy preserving the property that $p_-$ is the image of $p_+$ under the negative flow of $V$.  Applying this result to an isotopy passing $p_-$ past the point where the ascending manifold of $X$ intersects $\partial_-W'=f^{-1}(a +\epsilon)$ realizes the desired handle slide.

   \begin{figure}
\begin{center}
\scalebox{.7}{\includegraphics{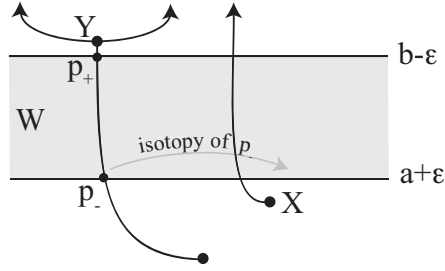}}
\caption{ Weinstein homotopy in $W'$ realizing a handle slide on $W$.}\label{fig:teleport}
\end{center}
\end{figure}

After constructing the homotopy described above,  some curve isotopic to $C$ is contained in a level set of $f_{1/3}$, and we can clearly arrange  this to be  $f_{1/3}^{-1}(1/2)$.  A further isotopy of $f_{1/3}$ near this level set ensures that $df_{1/3}(V_{1/3})=1$ near $C'$.  This immediately allows us to parametrize a neighborhood $\nu$ of $C'$ as $\nu = [1/2-\epsilon,1/2+\epsilon] \times C'$ so that in $\nu$, $f_{1/3}$ is projection onto the first factor $\zeta \in [1/2-\epsilon,1/2+\epsilon]$, $V_{1/3} =\partial_\zeta$, and $\beta = e^{l \zeta} d\zeta \wedge d\theta$, for some constant $l > 0$ and some angular coordinate $\theta \in [0, 2\pi]$ on $C'$. (Technically, we now have this result at time $t$ slightly greater than $1/3$, but we may reparameterize so that this has been achieved by time $t=1/3$.)

 Now consider an ambient isotopy $\psi_t : \nu \to \nu$ defined in the coordinates above by $\psi_t(\zeta,\theta)=(\zeta,\theta+h_t(\zeta))$ where $h_t: [1/2-\epsilon,1/2+\epsilon] \to [0,2\pi]$, for $t \in [1/3,2/3]$, is a family of functions as in Figure~\ref{fig:twist}. 
\begin{figure}
\label{fig:twist}
\labellist
\tiny\hair 2pt
\pinlabel $0$ [r] at 3 9
\pinlabel $2\pi$ [r] at 3 108
\pinlabel $\frac{1}{2}-\epsilon$ [t] at 6 7
\pinlabel $\frac{1}{2}$ [t] at 85 7
\pinlabel $\frac{1}{2}+\epsilon$ [t] at 164 7
\pinlabel $t=\frac{1}{3}$ [b] at 130 8
\pinlabel $t=\frac{1}{2}$ [b] at 140 56
\pinlabel $t=\frac{2}{3}$ [t] at 149 108
\endlabellist
\centering
 \includegraphics{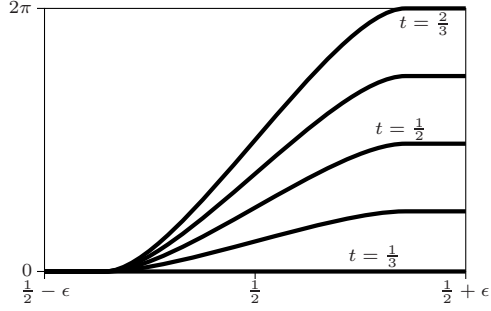}
 \caption{The functions implementing a Dehn twist.}
\end{figure}
Note that $\psi_0$ is the identity, that $\psi_1$ is a Dehn twist about $C'$, and that $\psi_t^* \beta = \beta$ and $\psi_t^* f_{1/3} = f_{1/3}$. Furthermore, $\psi_t^* (V_{1/3}) = V_{1/3}$ near $\partial \nu$. So now we define our homotopy for $1/3 \leq t \leq 2/3$ as: $f_t = f_{1/3}$ and $V_t = \psi_t^*(V_{1/3})$.

Thus, at time $t=2/3$ we have $(\beta,f_{2/3},V_{2/3})=\tau_{C'}^* (\beta, f_{1/3},V_{1/3})$. Then, for $2/3 \leq t \leq 1$, let $(f_t,V_t) = \tau_{C'}^*(f_{1-t},V_{1-t})$, and we are done.

\end{proof}

\begin{remark}

Lemma~\ref{lem:WeinsteinHomotopy} is stronger than required.  In the case that the original Morse function has $n$ index $0$ critical points which are fixed by $\phi$, the lemma produces a Weinstein homotopy on the  cobordism defined by deleting a neighborhood of each point.  One natural source of such a cobordims is the presence of a  transverse link $\mathcal{T}$ in the contact manifold.  Any such link may be transversely braided with respect to the open book \cite{Pav}, and one may take the intersections of $\mathcal{T}$ with the pages as the index $0$ critical points in the  Morse functions $f_t$.  

\end{remark}

\section{Contactomorphism}\label{sec:contacto}

We now prove Theorem~\ref{thm:contacto}, which we restate here for the sake of readability. Recall that  $W = (0,\infty) \times S^1 \times S^1$ with coordinates $x \in (0,\infty)$, $y, z \in S^1$, and with contact structure $\xi_W = \ker (dz + x\; dy)$. We are given some $(M,\xi,B,\pi)$.

\begin{thm1.1}
 There is a $2$-complex $\mathrm{Skel} \subset M$ with the property that, after modifying $\xi$ by an isotopy through contact structures supported by $(B,\pi)$, the complement $(M \setminus (\mathrm{Skel} \cup B), \xi)$ is contactomorphic to a disjoint union of $n=|B|$ copies of $(W,\xi_W)$.
\end{thm1.1}

\begin{proof}
 For the first claim, use Proposition~\ref{prop:contmob} to isotope $\xi$ and then produce a Morse structure $(F,V)$ on $(M,B,\pi)$ compatible with $\xi$. This then gives $\mathrm{Skel}$, and we now claim that each component of $(M \setminus (\mathrm{Skel} \cup B), \xi)$ is contactomorphic to $(W,\xi_W) = ((0,\infty) \times S^1 \times S^1, \ker (dz + x dy))$, via a contactomorphism taking $\pi$ to $z$ and $V$ to $x \partial_x$.

 To see this, note that there is one component of $M \setminus (\mathrm{Skel} \cup B)$ for each component of $B$. Fix one such component $Y$. We define a contactomorphism $\Psi: Y \to W$ as follows. First we define $\Psi$ on an open neighborhood $U$ of the relevant component of $B$, using local solid torus coordinates $(\rho,\mu,\lambda)$ as in Definition~\ref{def:mscs}, and coordinates $(x,y,z)$ on $W$:
 \[ \Psi(\rho,\mu,\lambda) = ((1/\rho^2),\lambda,\mu) \]
 Direct calculation verifies that $\alpha = (1/\rho^2) d\lambda + d\mu$ becomes $dz + x dy$, $V= -(\rho/2) \partial_\rho$ becomes $x \partial_x$ and $\pi = \mu$ becomes $\pi = z$. Thus $\Psi$ behaves as advertised on $U$ and $\Psi(U)$. Now $\Psi$ is uniquely determined on the rest of $Y$ and $W$ by the requirement that $D\Psi(V) = x \partial_x$, since all of $Y$ can be reached from $U$ by flowing along $V$. Since $V$ is tangent to pages and $x \partial_x$ is tangent to constant $z$ slices, it is clear that we will have $\Psi^* z = \pi$ everywhere.
 
 It remains only to verify that $\Psi$ is a contactomorphism everywhere. Write $\Psi_* \alpha$ as $a dx + b dy + c dz$ where, in principle, $a$, $b$ and $c$ are functions of $x$, $y$ and $z$. Recall that there is a vector field $X$ on $M$ (the monodromy vector field) such that $d\mu(X) = \alpha(X) = 1$. Thus $dz$ and $\Psi_* \alpha$ agree on $X$ which is transverse to constant $z$ levels, so the function $c$ is identically $1$; i.e., $\Psi_* \alpha = dz + a dx + b dy$. The $1$--form $a dx + b dy$ is the restriction of $\Psi_* \alpha$ to constant $z$ levels, i.e., ``pages'', so now we show that this restriction must be simply $x dy$. To see this, first note that $D\Psi(V) = x \partial_x$ is Liouville (on each page) for $dx \wedge dy = \Psi_* d\alpha$ on $\Psi(U)$, while $x \partial_x$ is Liouville for $dx \wedge dy$ on all of $W$ and $D\Psi(V)$ is Liouville for $\Psi_* d\alpha$ on all of $W$. Since all of $W$ is reached from $\Psi(U)$ by flowing along $D\Psi(V)=x \partial_x$, this implies that, restricting to pages, $\Psi_* d\alpha = dx \wedge dy$. But then, again restricting to pages, $\Psi_* \alpha = \imath_{D\Psi(V)} \Psi_* d\alpha = \imath_{x \partial_x} dx \wedge dy = x dy$.
\end{proof}

\begin{remark}
 As asserted in the introduction, $(M,B,\pi,\xi)$ is completely determined as a compactification of $\amalg^n W$ by the Morse diagram associated to the Morse structure $(F,V)$, and in fact any Morse diagram arises this way. This follows directly from  the fact that Proposition~\ref{prop:contmob} gives an explicit construction of $(M,B,\pi,\xi)$ starting from handle slide data recorded by the Morse diagram.
\end{remark}

In the remainder of the paper it is convenient to work with the contactomorphism $\Gamma = \Psi^{-1} : W \to Y$. Before concluding this section, we briefly describe an extension of $\Gamma$ which will prove useful for  Section~\ref{sec:tb}. 

Let $\widetilde{W} = [0,\infty] \times S^1 \times S^1 \supset W$ and extend $\Gamma$ to $\widetilde{W}$ so that $\{x=\infty\}$ maps onto a component of the binding and $\{x=0\}$ maps onto (a subset of) the skeleton. 

For each component $B_i$ of the binding and we let $\widetilde{W}_i$ be a corresponding copy of $\widetilde{W}$. Taking all of these together gives a surjective map $\widetilde{\Gamma}:\amalg_i \widetilde{W}_i\rightarrow M$. This  is a quotient map coming from an obvious equivalence relation on $\amalg_i \{x=\infty\}$ and a subtle equivalence relation on $\amalg_i \{x=0\}$. The map $\widetilde{\Gamma}$ factors through a space which we will call $\widetilde{M}$, a $3$-manifold with boundary defined by taking only the obvious equivalence relation on $\amalg_i \{x=\infty\}$,  so that the interior of $\widetilde{M}$ is naturally identified with the complement of the skeleton in $M$. In fact, $\widetilde{M}$ is nothing more than a disjoint union of compact solid tori, one for each component of the binding, but the important structure is the map $\widetilde{M} \to M$; we view $\widetilde{M}$ as a manifold-with-boundary compactification of $M \setminus \text{Skel}$, as distinct from the closed manifold compactification, which is $M$ itself.

\section{Front projections for Legendrian knots}\label{sect:front}

In this section we show how Morse structures may be used to define front projections of Legendrian knots.  

\begin{defn} If $\Lambda_i$ is a Legendrian curve in the complement of the skeleton and the binding, the \textit{front projection} $\mathcal{F}(\Lambda_i)$ is the result of flowing $\Lambda_i$ by $\pm V$ to the Morse diagram for the open book.  The front projection of a Legendrian knot $\Lambda\subset M\setminus B$ is the front projection of the curves $\Lambda\setminus (\Lambda\cap \text{Skel}(F,V))$.
\end{defn}

Recall that the tori $\amalg\mathcal{T}_i$ of the Morse diagram separate a neighborhood of the binding from the rest of the manifold.  When part of $\Lambda_i$ lies on the binding side of some $\mathcal{T}_i$, flow by $-V$ sends it to the Morse diagram, while curves on the opposite side of $\mathcal{T}_i$ flow via $V$ to the Morse diagram.

\begin{prop}\label{prop:frontproj} The front projection $\mathcal{F}(\Lambda)$ (after possibly perturbing $\Lambda$ by an arbitrarily small Legendrian isotopy) is a collection of smooth curves $\mathcal{F}(\Lambda_i)$ away from finitely many semicubical cusps.   Endpoints of $\mathcal{F}(\Lambda)$ occur in pairs with the same $t$-coordinate on opposite sides of paired trace curves. The tangent at each endpoint has slope $0$ and away from such endpoints, the slope of the tangent to $\mathcal{F}(\Lambda)$ lies in $(-\infty, 0)$.  The Legendrian knot $\Lambda$ may be recovered from $\mathcal{F}(\Lambda)$. 
\end{prop}

\begin{proof} After perturbing by a small Legendrian isotopy, we may assume that $\Lambda$ is disjoint from the binding and the critical points of the Morse structure on each page; that $\Lambda$ is transverse to the interiors of the $2$--cells of both $\mathrm{Skel}$ and $\mathrm{Coskel}$; and that $\Lambda$ is tangent to $V$ at only finitely many points. The transverse intersections with $\mathrm{Skel}$ and $\mathrm{Coskel}$ occur at discrete, distinct, values of $t$ which are not handle slide $t$-values.

In this case, the first assertion of the proposition follows from the analogous statement for front projections in $\mathbb{R}^3$ which generalizes immediately to $W$, via the contactomorphism $\Gamma$ defined above.  Recall that the vector field $V$ is identified with $x \partial_x$ under this contactomorphism, so that standard front projection in $W = \mathbb{R}^3/\sim$ corresponds to flowing along $V$ to the front diagram. This contactomorphism also immediately shows that the part of $\Lambda$ disjoint from $\mathrm{Skel}$ can be recovered from $\mathcal{F}(\Lambda)$.  Since $\Lambda$ is determined by $\mathcal{F}(\Lambda)$ outside a discrete set of points where it intersects $\text{Skel}(F,V))$,  continuity implies that it is determined everywhere.
 
If $\Lambda$ intersects the skeleton at a generic page, the  two end points which result from cutting $\Lambda$ at this intersection will flow under $V$ to opposite ends of the co-core dual to the intersecting core.  In the Morse diagram, $\mathcal{F}(\Lambda)$ therefore teleports across the corresponding paired trace curves. To see that the slope of $\mathcal{F}(\Lambda)$ is $0$ near each teleporting endpoint, it suffices to recall the extension of $\Gamma$ defined at the end of Section~\ref{sec:contacto} which maps the plane $x=0$ in $(\mathbb{R}^3, \xi_{\text{std}})$ to the skeleton.
\end{proof}

The discussion above implies that the Legendrian Reidemeister moves familiar from front projection in $\mathbb{R}^3$ carry over to the setting of open book front projections.  However there are a variety of other moves which change the planar isotopy type of the front projection, and these are explored in the  next section

To conclude, recall the statement of Theorem~\ref{thm:fronts} from the introduction:

\begin{thm1.4} 
 Let $\Lambda$ be a Legendrian link in $(M,\xi)$ that is disjoint from the binding,   transverse to $\mathrm{Skel}$, and tangent to $V$ at only finitely many points. Then the image of $\Lambda$ under the flow by $\pm V$ to $\amalg^n (0,\infty) \times S^1 \times S^1$ is a front on the Morse diagram. Furthermore, any front on this Morse diagram is the image of such a Legendrian $\Lambda$, and any two Legendrians with the same front are equal.
\end{thm1.4}

\begin{proof} Proposition~\ref{prop:frontproj}  establishes the first statement in the language developed since the introduction.  The fact that any abstract front is the front projection of  Legendrian knot follows from the fact that one may directly construct such a Legendrian by flowing via $\mp V$ back into the manifold for a time interval given by the slope on the front.  We note that  this allows us to easily construct examples, as any  curves satisfying the conditions of Definition~\ref{def:front} are in fact front projections of Legendrian knots; this is a useful property shared by  other instances of front projection (e.g., ($\mathbb{R}^3, \xi_{\text{std}})$ and universally tight lens spaces)\cite{BG}.

\end{proof}

\subsection{Fronts and Legendrian Isotopy}\label{sec:isotopymoves}

Figure~\ref{fig:moves} shows a collection of moves which change the surface isotopy type of the graph formed by $\mathcal{F}(\Lambda)$ and the trace curves on a Morse diagram; in case the figure is not self-explanatory, these moves will be described in more detail in the proof of Theorem~\ref{thm:frontmoves} below.  These moves, together with surface isotopy preserving the property of being fronts and  the ordinary Legendrian Reidemeister moves \cite{Et2} for front projections, make up the complete set of \textit{isotopy moves}.

\begin{thm1.5} Two Legendrian links in $(B, \pi, F, V)$ are Legendrian isotopic if and only if their fronts are related by a sequence of isotopy moves.  
\end{thm1.5}   

\begin{proof}


We need to prove two things: (1) each isotopy move on a front in fact corresponds to a unique Legendrian isotopy on the corresponding Legendrian, and  (2) any Legendrian isotopy $\Lambda_u$ between Legendrian links $\Lambda_0$ and $\Lambda_1$ can be perturbed so that there is a sequence of isotopy parameter values $0=u_0<u_1 < \ldots < u_m = 1$ such that, for each $i$, $\mathcal{F}(\Lambda_{u_i})$ is a front and differs from $\mathcal{F}(\Lambda_{u_{i-1}})$ by a single isotopy move. As we go through the proof below, we will describe each move in sufficient detail that (1) will be clear from the description of the move. 

To prove (2), we begin by noting that for any open cover $\{U_j\}$ of $M$, we can perturb the isotopy so as to ensure the following: there is a suitably small open cover $\{V_k\}$ of the domain $\amalg^p S^1$ of $\Lambda_0$ and a suitably small open cover $\{W_l\}$ of the parameter space $[0,1]$ such that, for each $u_* \in [0,1]$, there is at most one $W_l \times V_k$, with $u_* \in W_l$, on which $\Lambda_u$ is not independent of $u$, and the image of $W_l \times V_k$ is contained in exactly one $U_j$; i.e., the motion of $\Lambda_u$ happens entirely inside $U_j$. Furthermore, we may assume sufficient transversality so that, for all but finitely many values of $u$, $\Lambda_u$ satisfies all the genericity conditions spelled out in the first paragraph of the proof of Proposition~\ref{prop:frontproj}. At those finitely many values of $u$, exactly one of these genericity conditions will be violated, and this violation will occur in a transverse fashion. (To properly set up this transversality one should work with the Legendrian isotopy as a map of $[0,1] \times \amalg^p S^1$ into $[0,1] \times M$.) Both the statement about the isotopy being constant outside small $W_l \times V_k$'s and the transversality statement follow from the fact that Legendrian curves and Legendrian isotopies have as much local flexibility as curves and isotopies of curves in $\mathbb{R}^2$, thanks to Darboux's theorem and the standard theory of fronts and Legendrians in $\mathbb{R}^3 \to \mathbb{R}^2$.

The genericity conditions from the proof of Proposition~\ref{prop:frontproj}, which will be violated transversely, are as follows:
\begin{itemize}
 \item $\Lambda$ is transverse to the interiors of the $2$--cells of $\mathrm{Coskel}$. 
 \item $\Lambda$ is disjoint from the binding.
  \item $\Lambda$ is transverse to the interiors of the $2$--cells of $\mathrm{Skel}$.
 \item $\Lambda$ is disjoint from critical points of the Morse structure on each page.
  \item The transverse intersections with $\mathrm{Skel}$ and $\mathrm{Coskel}$ do not occur at handle slide $t$-values.
\end{itemize}

These are ordered so as to correlate with moves shown in Figure~\ref{fig:moves}.

We will see that each move addresses the failure of one of these conditions. Because of our assumption about how the isotopy changes with respect to a cover of $M$, we can always assume that the transverse violation of each condition occurs in some standard chart, so we simply present a standard model for each failure and show that it corresponds to one of the isotopy moves.

   \begin{figure}
\begin{center}
\scalebox{.75}{\includegraphics{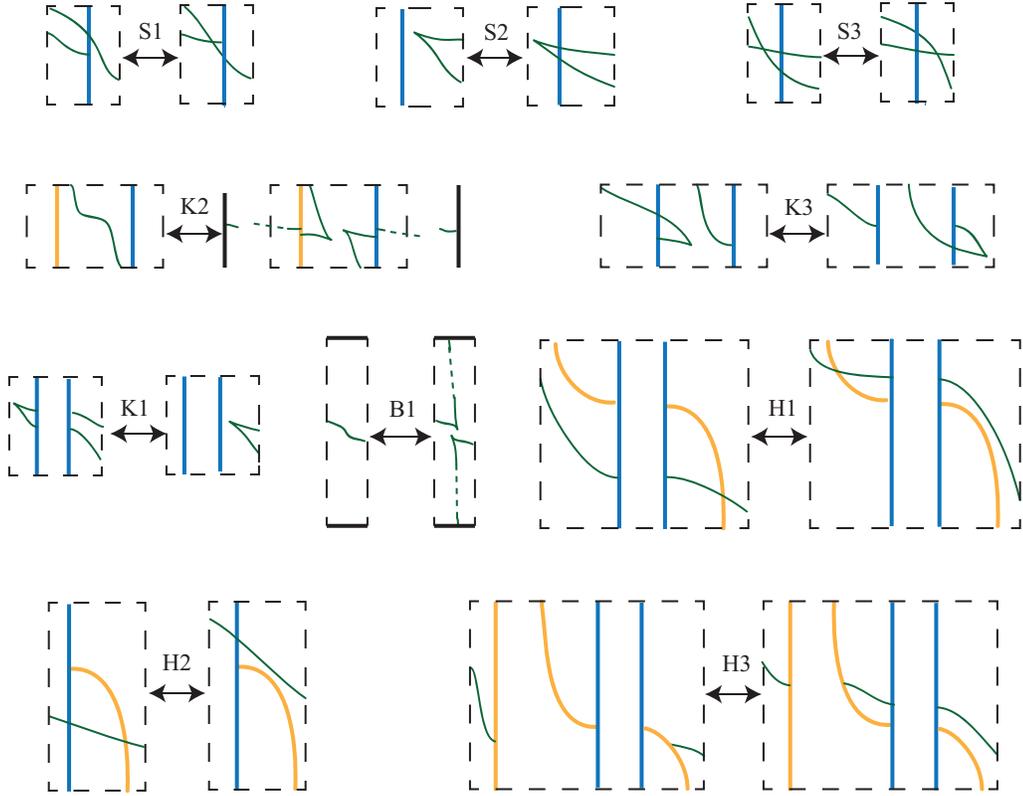}}
\caption{  These moves, together with their reflections preserving the negative slope of $\mathcal{F}(\Lambda)$, join the Legendrian Reidemeister moves to give a complete set.  On each partial Morse diagram, $t$ is vertical and $s$ is horizontal. Note that moves K2 or B1 are not completely local, as  K2 potentially creates  many teleports and the right hand picture of B1 may have lots of new crossings.}\label{fig:moves}
\end{center}
\end{figure}

\subsubsection{Isotopies within standard solid tori}

Recall that the solid torus $W$ is a quotient of $\big(\mathbb{R}^3, \ker (dx+xdy) \big)$. Front projection to an $x=c$ plane is classically understood, and the contactomorphism $\Gamma$ identifies the Morse diagram torus $ \mathcal{T}_i$  with the image of such a plane.  This immediately implies that Legendrian isotopies confined to a single component of $M\setminus \text{Skel}$ can be viewed as the image of Legendrian isotopies in the standard contact $\mathbb{R}^3$.  Thus the list of isotopy moves must contain the ordinary Legendrian Reidemeister moves for front projections, respecting cusps and teleporting endpoints.  However, Legendrian isotopy can give rise to new moves changing the planar isotopy type of this graph formed by $\mathcal{F}(\Lambda)$ and the trace curves.  

\textbf{Move S1:} Translation in the $t$ direction is  a Legendrian isotopy.  If  there is a segment of $\mathcal{F}(\Lambda)$ which crosses a trace curve at height $t_1$ and another segment which teleports across it at height $t_1\pm \epsilon$, any  local isotopy which would move the crossing point past $t_1\pm \epsilon$ in the absence of the teleporting point is allowed to pass the crossing through the teleporting point.

  \textbf{Move S2:}Translation in the $s$ direction is  also a Legendrian isotopy, but such a translation may result in an instantaneous tangency between $\Lambda$ and the interior of a $2$-cell of Skel.  In the front projection, this appears as cusp may passing through a trace curve, and we denote this move by $S2$. 
 
 \textbf{Move S3:} Since translation in the $s$ direction is a Legendrian isotopy, a crossing on the front projection can pass through a trace curve.  Note that although this move changes the planar isotopy type of the graph on the Morse diagram, it does not represent a failure of $\Lambda$ to be generic.

In fact, the move H1 shown in Figure~\ref{fig:moves} also arises from an isotopy restricted to a single solid torus, but we include this discussion with the rest of the  handle slide moves below.

\subsubsection{Isotopy across the binding}

The requirement that $\Lambda$ be disjoint from the binding  leads to a new move on the front projection associated to a Legendrian isotopy passing $\Lambda$ through $B$.  Every component of the binding has a standard contact neighborhood, and in fact, this move has already appeared in the literature in the case of front projections of Legendrian knot in universally tight lens spaces.  See, for example \cite{BG}.

\textbf{Move B1:} Let $\gamma$ be a nearly vertical segment of $\mathcal{F}(\Lambda)$.  Recall that we can ensure an arbitrarily negatives slope by  isotoping the corresponding segment of $\Lambda$ to lie sufficiently close to the binding.  Then $\gamma$ may be replaced by its approximate vertical complement, connected to $\mathcal{F}(\Lambda)\setminus \gamma$ by a pair of cusps.   
 
\begin{rem}A useful consequence of  B1 is that $\Lambda$ is always Legendrian isotopic to some $\Lambda'$ contained in the complement of a fixed page, and further, that we may easily construct a front projection  $\mathcal{F}(\Lambda')$ by applying this move  to all intersection of $\mathcal{F}(\Lambda)$ with some curve of fixed $t$ value. 
\end{rem}

\subsubsection{Isotopy across the skeleton}
In the remainder of the proof, we examine the rest of the way in which an isotopy can violate one of the genericity conditions established above.  Such violations occur when an isotopy moves a segment of $\Lambda$ between two components of $M\setminus \text{Skel}$. 

\textbf{Move K1:} 
We first consider the effect of pushing an arc of $\Lambda$ across the skeleton, violating the generic conditions that $\Lambda$ be transverse to the interior of the $2$-cells of Skel.  Recall that each solid torus  is contactomorphic to a quotient $W$ of the subset of $\mathbb{R}^3, \xi_{\text{std}})$ defined by $0< x \leq 1$, where a sequence whose limit lies on $x=0$ maps to a sequence whose limit lies on the skeleton.  In $\mathbb{R}^3$, an arc lying completely in the region where $x>0$ can approach $x=0$ only if its front projection has a cusp; thus in the open book case, the front projection of a Legendrian arc approaching the skeleton shows a cusp approaching a trace curve.  The skeleton separates two solid tori, playing the role of the $x=0$ plane for each of them, so after the arc has passed through the point of tangency, we can understand its front projection by considering a Legendrian arc in $\mathbb{R}^3$ with both endpoints on $x=0$ and its interior lying in $x>0$.  Again, the front projection of such an arc must have a cusp, and if we restrict our isotopy sufficiently, then there will be only one such cusp.  Thus Move K1 teleports a cusp across a trace curve.  

\begin{remark} Analyzing the characteristic foliation offers an alternative proof that the front projection of a Legendrian arc crossing the skeleton  results in a teleporting cusp; we will use this perspective in the proof of the next move and in the study of isotopies across handle slides.
\end{remark}

\textbf{Move K2:} 
Next, we consider the effect of a Legendrian isotopy passing $\Lambda$ through an index $0$ critical point.  The flowlines on each page are the leaves of the characteristic foliation of  $\ker \alpha$, so the union of index $0$ critical points on all the pages is easily seen to form a transverse knot in $M$.  Some tubular neighborhood of any transverse knot is standard, which implies the existence of a Legendrian isotopy moving $\Lambda$ across $c_0\times S^1$.  If we assume that the initial segment is disjoint from the skeleton, then the new connecting segment will intersect the piece of the skeleton above each core curve twice.  On the front projection, this corresponds to replacing a nearly horizontal segment of $\mathcal{F}(\Lambda)$ by a nearly horizontal curve which teleports across every pair of trace curves twice and connects to the rest of the projection via a pair of cusps.  This move is analogous to Move B, reflecting the fact that both involve isotopy across a transverse knot.

\textbf{Move K3:}
Finally, consider an isotopy in which $\Lambda$ crosses an index $1$ critical point $c_1$; after possibly applying  K1, we may assume that $\gamma\subset \Lambda$  intersects both the skeleton and the co-skeleton once near $c_1$.  As in the discussion of K2 above, we note first that the union of index $1$ critical points with the same label form a transverse curve in $M$, and thus there is some Legendrian isotopy passing $\gamma$ through this curve.  We argue first that the image of $\gamma$ after some isotopy of this type intersects the core and co-core on the other side of the critical point as shown in move $K3$.  

In fact, our argument that the front move K3 represents a valid Legendrian isotopy is essentially one of continuity.  As long as $\gamma$ remains distinct from the critical point $c_1$, the ``before" and ``after" pictures in K3 depict Legendrian knots.  As $\gamma$ approaches $c_1$ from either side, the slopes of the cusps approach $0$ and the cusps themselves approach the trace curve.  The pointwise limit of both front diagrams is a curve with an everywhere negative slope which teleports across the paired trace curves at the point where the cusps approach.  This is the front projection of a Legendrian curve interpolating between the initial and final figures in the move K3.

Having established that K3 represents a valid front projection move, we would like to conclude that it models any generic isotopy of $\Lambda$ across an index $1$ critical point.  Consider a $t$-interval sufficiently small that the characteristic foliation near $c_1$ varies only by isotopy, and consider the Lagrangian projection of the isotopy shown in K3.  See Figure~\ref{fig:K3}.  Observe that there is a single point of tangency between the projection of $\gamma$ and the characteristic foliation which corresponds to the cusp.  We may assume such points of tangency are isolated, and perhaps after conjugating by a pair of K1 moves, we may restrict the isotopy under consideration to a sufficiently small neighborhood in $M\times I$ to include the passage of just a single such tangency across $c_1$, recovering K3 as desired.

Note that this move has two versions depending on whether the initial cusp is an ``up" or ``down" cusp; this property is preserved by the isotopy.

 \begin{figure}
\begin{center}
\scalebox{.75}{\includegraphics{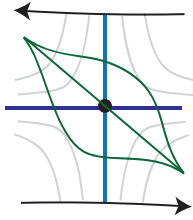}}
\caption{ Passing a segment of $\Lambda$ across an index $1$ critical point.}
\label{fig:K3}
\end{center}
\end{figure}

\subsubsection{Handle slide moves}

Although the vector field $V$ (and hence, the contact structure) changes continuously in the $t$-direction, the evolution of the flowlines is not continuous as the $t$ parameter passes through a handle slide value.   For segments of $\Lambda$  away from the handle-sliding curves, the front projection changes only by the moves described above, so in this final section we study the effect on the front projection of an isotopy passing $\Lambda$ through the cores and co-cores on a page where a handle slide occurs.  If  $t_*$ is a handle slide value, then on $\Sigma_{t_*}$ there is a flowline from an  index $1$ critical point $c_1$ to an index $1$ critical point $b_1$.  Let $X$ denote the union of this flowline together with the descending manifold from $c_1$ and the ascending manifold from $b_1$.      

Suppose that $u$ parameterizes a Legendrian isotopy which passes a segment $\gamma\subset\Lambda$  transversely through $X$ in the complement of the critical points, and let $u_*$ denote the value of the isotopy parameter  $\gamma_{u_*}$ meets  $X$.  Since transversality of curves is an open condition,  there exist  non-empty intervals $[u_1,u_2]\owns u_*$  and $[t_1,t_2]\owns t_*$ such that $\gamma(u)$ is transverse to the curves of the characteristic foliation on $\Sigma_t$ for all $u\in [u_1,u_2]$ and $t\in [t_1, t_2]$. Up to isotopy preserving intersections with the characteristic foliation, it follows that the projections of $\gamma_{u_i}$ to $\Sigma_{t_i}$ agree with the numbered arcs in Figure~\ref{fig:charfol}.   Treating a given arc as fixed, we examine how its intersections with flowlines of $V$ change as the handle slide occur.  The front projection of the numbered arc before and after the isotopy may be read off from each of these pictures; t arcs $1$ and $2$ correspond to move H3, arc $3$ corresponds to move H1,  and arcs $4$ and $5$ correspond to move H2. 

   \begin{figure}
\begin{center}
\scalebox{.55}{\includegraphics{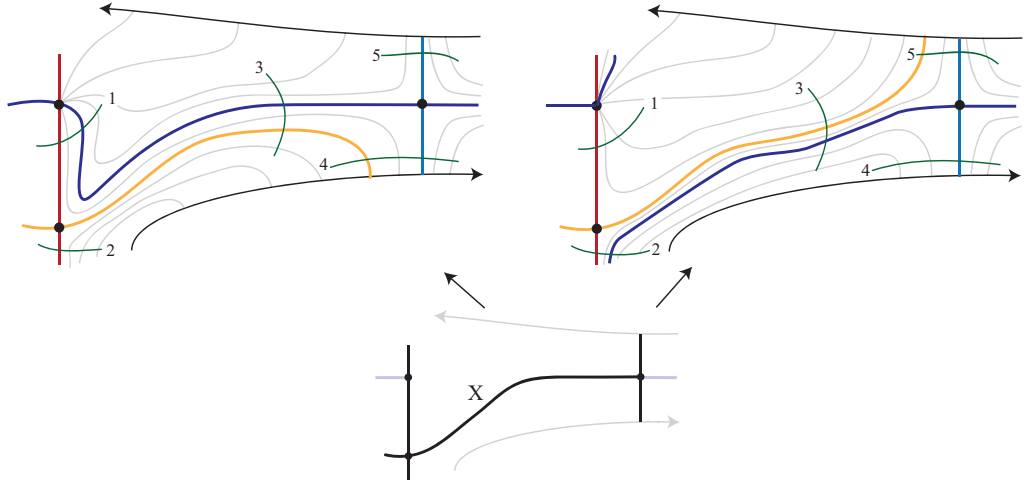}}
\caption{ Moves H2, H2, and H3 come from comparing the intersection of a fixed numbered arc with the characteristic foliation before and after a handle slide.}\label{fig:charfol}
\end{center}
\end{figure}

\end{proof}

\section{Computing the Thurston-Bennequin number}\label{sec:tb}
As an application of the front projection introduced in the preceding section, we define an algorithm to construct a Seifert surface for a nullhomologous Legendrian knot $\Lambda$ and use this to compute the Thurston-Bennequin number of $\Lambda$ from its front projection  $\mathcal{F}(\Lambda)$.  In the remainder, we assume that $\Lambda$ is oriented. Throughout this section, we let $\Sigma$ denote the abstract compact surface homeomorphic to the embedded page $\Sigma_0$.
  
\subsection{Detecting nullhomologous knots}\label{sec:nullhom}
In this section we explain how the front projection $\mathcal{F}(\Lambda)$ can be used to detect when  $\Lambda$ is nullhomologous in the closed $3$-manifold.  The existence of move B1, which isotopes a Legendrian curve across the binding, implies that we may assume $\Lambda$ is disjoint from $\Sigma_0$; in fact, we will assume that $\Lambda$ lies in $\Sigma\times [\epsilon,1-\epsilon]$ for some $\epsilon>0$.

Let $P=\{p_1, \dots p_n\}$, where $p_i$ is a point on the $i^{th}$ boundary component of $\Sigma$ in the complement of  the co-cores for $(f_0,V_0)$. The inclusion of $\Sigma$ into $M$ induces a surjective map $H_1(\Sigma)\rightarrow H_1(M)$.  In order to describe the kernel $K$ of this map, we consider a different inclusion-induced map, $i_*:H_1(\Sigma)\hookrightarrow H_1(\Sigma,P)$. Note that $i_*$ is injective, since $H_1(P)=0$. The monodromy $\phi$ induces a map $\phi_*:H_1(\Sigma,P)\rightarrow H_1(\Sigma, P)$,  and $K$ is equal to  $\text{im}(\phi_*-\text{id})\subset H_1(\Sigma)$. Note that  $\text{im}(\phi_*-\text{id})\subset H_1(\Sigma)$ because $\phi$ fixes $P$.

 With the assumption that $\Lambda\cap \Sigma_0=\emptyset$, we may define a projection $\pi: \Lambda\rightarrow \Sigma_0$. It follows that $\Lambda$ is nullhomologous in $M$ if and only if $[\pi(\Lambda)]\in K\subset H_1(\Sigma_0)$, and we next explain how to detect the homology class of  $\pi(\Lambda)$ from $\mathcal{F}(\Lambda)$.

First, we choose a set of generators for $H_1(\Sigma_0, P)$ as follows:  once oriented, the cores form a basis for $H_1(\Sigma_0)$, and we extend this to a generating set for $H_1(\Sigma_0, P)$ by including a flowline from the index $0$ critical point  to each marked point $p_i$ on the boundary.  The concatenation of any pair of flowlines, one with its orientation reversed, provides the desired generating set.  For convenience, choose the marked point $p_i$ so that the planar presentation of each component of the Morse diagram has $p_i\times [0,1]$ as its left (and right) vertical edge.  For later convenience, we will orient the left edge up and the right edge down; this will allow us to formally treat the edges as if they were  paired trace curves.

The  co-cores are dual to this basis for $H_1(\Sigma_0)$, so counting signed intersections with the co-cores detects the first homology class of any closed curve on $\Sigma_0$.  Intersections with co-core curves are in bijection with intersections between $\mathcal{F}(\Lambda)$ and the trace curves, so with a little bookkeeping, we may compute the homology class of $\pi(\Lambda)$ from the front projection.  

Each oriented co-core has an incoming and an outgoing end; on the Morse diagram, orient the trace curve corresponding to the former upwards and the trace curve corresponding to the latter downwards.   It follows that if $\Lambda$ lies in $\Sigma\times [0,a]$ for $a$ less than the smallest handle slide $t$-value, counting its signed intersections with the trace curves  computes  its homology class in $H_1\big(\Sigma\times [0,1]\big)$ in terms of the initial basis. (Note that handle teleporting points are not intersections, since  teleporting across a pair of trace curves implies that the corresponding segment of $\Lambda$ is disjoint from the corresponding co-core.)

In order to extend this technique to $\Lambda\subset \big(\Sigma \times[\epsilon, 1-\epsilon]\big)$, we need to know what homology class to assign to an intersection between $\mathcal{F}(\Lambda)$ and  a trace curve after some number of handle slides  have occurred.  Suppose that handle $B$ slides over handle $A$; since the trace curves come from co-cores, this appears one of the trace curves labelled $A$ teleporting from one trace curve labelled $B$ to the other trace curve labelled $B$. Suppose also that this $A$ trace curve and the $B$ trace curve at which $A$ enters the teleport have the same orientation in the Morse diagram.  Then after the handle slide the corresponding trace curves come from co-cores  dual to the classes $A$ and $A+B$, respectively.   This is shown in the example on the right in Figure~\ref{fig:label}.  On the other hand, if $A$ and $B$ have opposite orientations at the teleport, then the resulting co-cores will be dual to the classes $A$ and $A-B$. 

   \begin{figure}
\begin{center}
\scalebox{.8}{\includegraphics{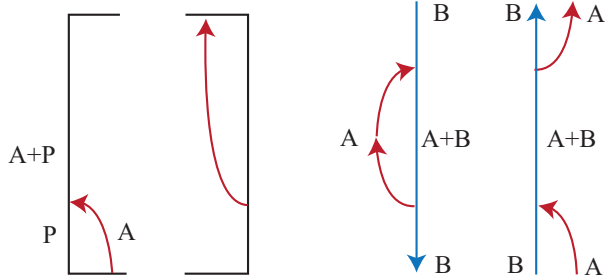}}
\caption{ Labeling the trace curves and vertical edges by homology class.  }\label{fig:label}
\end{center}
\end{figure}

Similarly, label the initial arc from $p_i$ to $p_j$ by the homology class $P_{ij}\in H_1(\Sigma_0,P)$.  If a co-core labelled $A$ crosses the marked point $p_i$ --- on the trace diagram, this looks like a ``handle slide" across the paired left and right edges of the planar Morse diagram --- the corresponding arc on the surface changes from the class $P_{ij}$ to the class $P_{ij}+A$, just as if the edges were in fact trace curves. Similarly, we label the edge of each component of the Morse diagram by the half-edge $P_i$, with the understanding that a class in $H_1(\Sigma, P)$ is represented by a difference of edges.  To compute the kernel, we again take the difference between top and bottom labels, but this time, of  differences of edges.

Split each trace curve and diagram edge into intervals separated by  teleporting points of the trace curves  and label each interval with the appropriate homology class. 
In order to determine the homology class of  $\Lambda$ in $\Sigma\times [0,1]$, it suffices to sum the signed labels of the intersection points between $\mathcal{F}(\Lambda)$ and the trace curves.

To determine the homology class of the same knot, now viewed as a submanifold of $M$, we compute the image of this under the quotient by $\text{im}(\phi_*-\text{id})$.    We depict the Morse diagram as planar for convenience, but of course, the top and bottom edges are identified:  $\phi_*$ applied to the $t=0$ label is the $t=1$ label.  Thus the differences between  the labels at the top and bottom of each trace curve and vertical edge generate $\text{im}(\phi_*-\text{id})$. 

   \begin{figure}
\begin{center}
\scalebox{.6}{\includegraphics{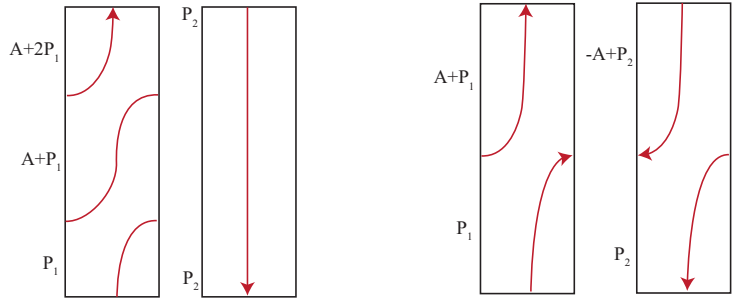}}
\caption{Two Morse diagrams for $L(2,1)$.  On the left, we compute that $K$ is generated by
$(P_1-P_2)-\big(A+2P_1-P_2\big)$, and on the right, by $(P_1-P_2)-\big(A+P_1)-(-A+P_2)\big)$. }\label{fig:l21}
\end{center}
\end{figure}

This discussion establishes the following lemma:

\begin{lemma} With  notation as above, the following statements hold:
\begin{enumerate}
\item $H_1(\Sigma \times [0,1]) \cong H_1(\Sigma)$ is  generated by the trace curve labels along the bottom edge. 
\item $K = \text{im}(\phi_*-\text{id})$ is generated by the differences between the top and bottom labels, including the labels of the vertical edges.   Although $P_i$ labels appear at the top and bottom, they will necessarily cancel out in the differences. 
\item $H_1(M)\cong H_1(\Sigma) \slash K$. 
\item  $[\pi(\Lambda)] \in H_1(M)$ is equal to the signed sum of labels of intersections between $\mathcal{F}(\Lambda)$ and the trace curves. If this signed sum is $0$ then $\Lambda$ is nullhomologous in $\Sigma \times [0,1]$.  If this signed sum lies in $K$, then $\Lambda$ is nullhomologous in $M$.
\end{enumerate}
\end{lemma}

\begin{example} We illustrate this with an open book with punctured torus pages,  shown in Figure~\ref{fig:ex3} with the trace curves assigned their homological labels. The signed sum of the intersection labels between $\mathcal{F}(\Lambda)$ and the trace curves is $-2A+B$.  The top and bottom labels agree for all but the downward pointing blue trace curve, so $\text{im}(\phi_*-\text{id})$ is generated by the difference between these two labels: $A$.  $H_1(M)\cong \mathbb{Z}$, and $\Lambda$ generates the first homology of the manifold.

   \begin{figure}
\begin{center}
\scalebox{.8}{\includegraphics{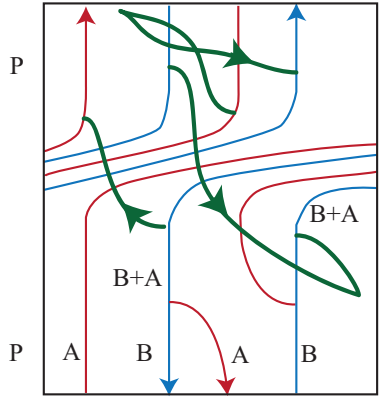}}
\caption{ $[\Lambda]$ is  dual to the co-core labeled $A$.   }\label{fig:ex3}
\end{center}
\end{figure}

\end{example}

\subsubsection{The total writhe}\label{sec:writhe}
Once we determine  that a knot contained in the cylinder $\Sigma\times [\epsilon,1-\epsilon]$ is nullhomologous, we can compute its Thurston-Bennequin number.  In order to do so, we will define a quantity called the \textit{total writhe} of the front projection.  This is closely related to the ordinary writhe of a diagram, but it also counts some crossings between $\mathcal{F}(\Lambda)$ and the trace curves.

Suppose first that $[\pi(\Lambda)]=0\in H_1(\Sigma)$.  
 Assign signs to each teleporting endpoint of $\mathcal{F}(\Lambda)$ as shown in Figure~\ref{fig:signs} and divide each trace curve into intervals separated by these teleporting endpoints of $\mathcal{F}(\Lambda)$ and by teleporting trace curves.  (For clarity, we will refer to the latter as handle slides in the remainder of this section.)  Assign multiplicities to each interval as follows.  Assign the bottom interval of each trace curve multiplicity $0$. 

If  there are no handle slides,  define the multiplicity of any interval to be the sum of the multiplicity of the interval below and $1$ times the sign of the teleporting endpoint separating them.  Observe that corresponding intervals on paired trace curves will have multiplicities of the same absolute value, but of opposite signs.   

   \begin{figure}
\begin{center}
\scalebox{.7}{\includegraphics{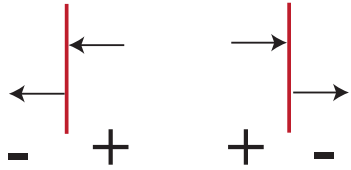}}
\caption{ Incoming arrows are labeled ``$+$" and outgoing arrows are labelled ``$-$".    }\label{fig:signs}
\end{center}
\end{figure}

When a handle slide occurs, add the multiplicities of the two converging branches to define the multiplicity of the interval above the trivalent point.  This, together with the condition on matching multiplicities of paired trace curves, determines the multiplicity of the two branches emanating from the paired trivalent point. See Figure~\ref{fig:multex}.

\begin{figure}
\begin{center}
\scalebox{.65}{\includegraphics{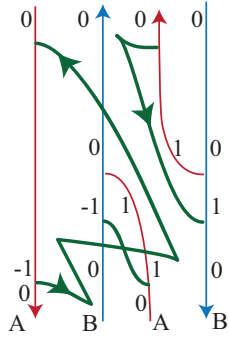}}
\caption{ A Morse diagram for a nullhomologous knot indicating  the multiplicities of each trace curve interval.}\label{fig:multex}
\end{center}
\end{figure}

Let $\text{P}$ and $\text{N}$ denote the total number of positive and negative self-crossings of $\mathcal{F}(\Lambda)$, respectively.  Let $T$ denote the sum of signed  crossings between $\mathcal{F}(\Lambda)$ (viewed as the over strand) and the upward-oriented trace curves (viewed as the under strand) and weighted by the multiplicity of each trace curve interval where a crossing occurs.

\begin{defn} When $\Lambda$ is nullhomologous in $\Sigma\times[0,1]$, the \textit{total writhe} of $\mathcal{F}(\Lambda)$ is the sum 
\[ \text{W}\big(\mathcal{F}(\Lambda)\big)=P-N+T.\]
\end{defn}

In the case that $[\pi(\Lambda)]\neq 0 \in H_1(\Sigma)$, we introduce a link $X$ with the properties that $X$ and $\Lambda$ are homologous, and $X\cup \Lambda$ is nullhomologous in $\Sigma\times[0,1]$, as above.  In fact, the link $X$ may be easily represented by adding curves with one of two standardized forms to the Morse diagram, as we describe next:
\begin{enumerate}
\item A Type 1 pair consists of a curve lying in the $t$-interval $(1-\epsilon, 1]$ together with its orientation-reversed translation to the $t$-interval $[0,\epsilon)$. The endpoints of each curve should form a teleporting pair.
\item Each member of a Type 2 pair  is a curve which begins on the left-most trace curve on a component of the Morse diagram at some $t$-value  in $(1-\epsilon, 1)$, travels left nearly to the edge of the diagram, travels vertically down to a $t$-value in $(0, \epsilon)$, and then travels right to terminate on the left-most trace curve.  The two curves of a Type 2 pair should lie on different components of the trace diagram, and should have opposite orientations along their vertical segments.
\end{enumerate}

\begin{figure}
\begin{center}
\scalebox{.6}{\includegraphics{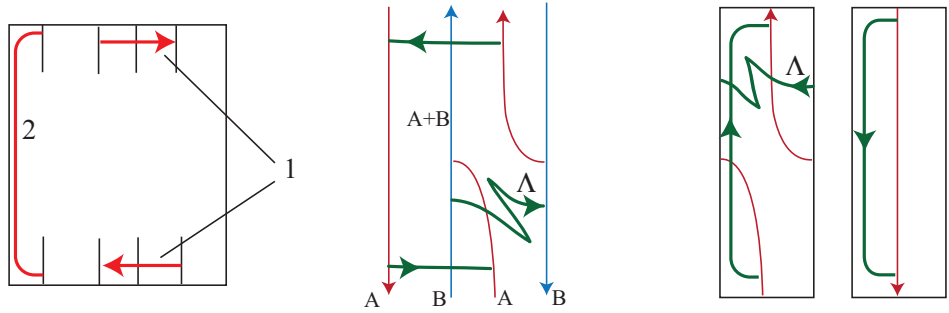}}
\caption{ Left: One curve of a Type 2 pair and both elements of  Type 1 pair.  Center and Right: Examples of $\mathcal{F}(\Lambda\cup X)$. }\label{fig:Xcurves}
\end{center}
\end{figure}

Let $\mathcal{F}(X)$ be a collection of curves of Types 1 and 2 such that $[\pi(X)]=-[\pi(\Lambda)]\in H_1(\Sigma)$, as computed by summing the labels of all intersection points with the labelled trace curves. That such a link exists is easily seen from the following argument: choose a Seifert surface for $\Lambda$ and cut it along its  intersection with $\Sigma_0\cup N(B)$ and pull the resulting surface into the interior of the cut manifold. Each pair of boundary components produced by the cut will consists of closed curves representing the homology classes $Y$ and $\phi_*(Y)$ near the top and bottom of the cylinder, respectively, and their front projections will form a Type $1$ pair.  When cutting yields a single boundary component, it will have the form $\big(\gamma\times (1-\epsilon)\big) \cup \big(\phi(\gamma)\times \epsilon\big) \cup \big( \{p, q\} \times [0,1] \big)$, where $\gamma$ is an arc connecting points $p$ and $q$ near $\partial \Sigma$.  The front projection of such a curve will appear as a pair of Type 2 curves on the Morse diagram.

Compute the total writhe as above, now using  the teleporting endpoints of $\mathcal{F}(X\cup \Lambda)$ to assign multiplicities to intervals of the trace curves.

\begin{defn} For $\Lambda$ nullhomologous in the mapping torus $M(\Sigma, \phi)$ and $\mathcal{F}(X)$ as above, let $P$ and $N$ denote the counts of positive and negative crossings between $\mathcal{F}(\Lambda)$ and itself, while $T$ denotes the the sum of signed crossings between $\mathcal{F}(\Lambda)$ and the trace curves weighted by the multiplicities induced by $\mathcal{F}(\Lambda\cup X)$.  Then the \textit{total writhe} of $\mathcal{F}(\Lambda\cup X)$ is the sum 
\[ \text{W}\big(\mathcal{F}(\Lambda\cup X)\big)=P-N+T.\]
\end{defn}

\begin{lemma}\label{lem:writhe} The total writhe of $\mathcal{F}(\Lambda \cup X)$ is independent of the choice of $X$.
\end{lemma}

It follows that the total writhe is in fact an invariant of the original front projection $\mathcal{F}(\Lambda)$, and consequently, we write $W(\mathcal{F}(\Lambda))$ instead.  The proof is at the end of the section. With this notation in hand, we recall the formula for computing the Thurston-Bennequin number stated in the Introduction:

\begin{thm1.6}
Let $\Lambda$ be a nullhomologous Legendrian knot. Then the Thurston-Bennequin number of $\Lambda$ is computed by
\[ \text{tb}(\Lambda)= \text{W}\big(\mathcal{F}(\Lambda)\big)-\frac{1}{2}|\text{cusps}|
.\]
\end{thm1.6}

\begin{example} The Morse diagram on the right in Figure~\ref{fig:Xcurves} describes an open book with annular pages whose monodromy is a single left-handed Dehn twist about the core: an overtwisted $S^3$.  In fact, $\Lambda$ is the boundary of an overtwisted disc.  There is a single negative crossing between $\Lambda$ and the trace curve on the first component of the diagram ---recall that the trace curve is the undercrossing strand.  However, the multiplicity of this interval of the trace curve is $-1$, as a consequence of the Type 2 pair, so the contribution to the total writhe is $T=1$.  Adding this to half the number of cusps computes that the Thurston-Bennequin number of $\Lambda$ is $0$.
\end{example}

\subsection{Seifert surfaces I}\label{sec:ss1}

  In order to prove Theorem~\ref{thm:tb}, we will need to construct a Seifert surface $S$ for $\Lambda$ and count the signed intersections between $S$ and a vertical translation $\Lambda'$ of $\Lambda$.  When $\Lambda$ is nullhomologous in the cylinder $\Sigma \times [0,1]$, the procedure for constructing a Seifert surface is somewhat simpler than in the case that $\Lambda$ is nullhomologous only in the mapping torus $M(\Sigma,\phi)$.  In this section, we assume that $[\pi(\Lambda)]=0\in H_1(\Sigma)$, and we defer the latter case to Section~\ref{sec:ss2}.

  Our approach mimics the classical construction in $\mathbb{R}^3$, with some extra  care required to deal with the fact that the front projection of $\Lambda$ will  generally not be an immersed loop.  A summary of our approach is as follows: cutting $M$ along its skeleton separates $\Lambda$ into properly embedded arcs.  We  connect the endpoints of these arcs via new segments lying on the boundary of    $\widetilde{M}:=\overline{M\setminus \text{Skel}}$; see the discussion at the end of Section~\ref{sec:contacto}). These segments are chosen to appear  in pairs which cancel when $\widetilde{M}$ is mapped back into $M$.  It thus suffices to build a Seifert surface for the link in $\widetilde{M}$.

\textbf{Step 1:}   Recall that $\widetilde{M}$ is a compactification of $M\setminus \text{Skel}$, so our link has a natural preimage in $\widetilde{M}$ consisting of properly embedded arcs.  In this step, with a little care, we will connect the endpoints of these arcs by additional segments in $\partial \widetilde{M}$.

Front projection collapses surfaces in $\partial \widetilde{M}$ to intervals, so as an aid to visualization, we replace each trace curve $T$ by a ribbon $T\times [-1,1]$.  Strictly speaking, the result is no longer a front projection, but the honest front projection may be recovered by projecting each ribbon to $T\times \{0\}$.

Recall that the definition of the total writhe involved first decomposing the trace curves into intervals bounded by teleporting endpoints and trivalent points of the trace pattern.  

\begin{lemma}\label{lem:intnumb} 
On each connected component of the trace pattern, the sum of signs of the teleporting endpoints is $0$
\end{lemma}

The proof is deferred until the end of the section.

We assigned each interval $T_i$ of $T$  a multiplicity $m_i(T)$, and we place $m_i(T)$ disjoint positively oriented segments in the ribbon neighborhood of $T_i$. According to the lemma, the endpoints of these segments can be connected to each other and to the teleporting endpoints of $\mathcal{F}(\Lambda)$ to yield a collection of  immersed closed curves $\mathcal{F}(\widetilde{\Lambda})$.  We claim that these connections may be chosen so as not to introduce any new crossings into the diagram; if an initial choice of connections creates crossings,  the oriented resolution of these crossings will satisfy the claim.

To complete the construction of $\widetilde{\Lambda}$, lift the added segments to disjoint segments on $\partial \widetilde{M}$ which connect to each other and to the points where $\Lambda_H$ intersects $\partial \widetilde{M}$.  Since the segments appear in oppositely oriented pairs on ribbons associated to paired trace curves, these lifts may be chosen so that the natural map of $\widetilde{M}$ onto $M$ glues cancelling segments.  This ensures that any Seifert surface for $\widetilde{\Lambda}$ in $\widetilde{M}$ will glue to a Seifert surface for $\Lambda$ in $M$, as desired.

\textbf{Step 2:}
Now apply the classical Seifert's Algorithm to $\mathcal{F}(\widetilde{\Lambda})$: take the oriented resolution at each crossing to transform $\mathcal{F}(\widetilde{\Lambda})$ into a collection of disjoint simple closed curves.  

Recall that the original curve $\pi(\Lambda)$ was nullhomologous in $\Sigma_0$.  If it  had nonzero winding number with respect to the index $0$ critical point $c_0$, then some of these curves will also share this property; although nullhomologous, they will not bound a  region on the front projection.

To remedy this, introduce additional simple closed curves to the link whose winding number with respect to the index $0$ critical point is the negative of that of $\Lambda$.  Furthermore, these curves may be chosen to lie sufficiently close to $c_0$ on pages with sufficiently large $t$-coordinate that no new crossings are introduced in the front projection.  The front projection of a curve contained on a fixed page and linking $c_0$ will appear on the front projection as a horizontal line, but it is important to note that it teleports at each trace curve, rather than crossing it.  After introducing these new components, the curves in the resolved front projection will bound discs and annuli.

Next, lift these to discs and annuli in $\widetilde{M}$.  This must be done with some care;  in the case of nested cycles $C_1$ inside $C_2$, note that the interior of the disc bounded by $C_2$ should be assumed to lie closer to the binding than $C_1$.  This condition is required by the fact that segments of $C_1$ may lie on the boundary of $\widetilde{M}$, preventing the interior of the disc bounded by $C_2$ from being pushed ``behind" $C_1$.  For later use, we further assume that a collar neighborhood of $\Lambda$ in each disc agrees with $\Lambda \times [r_0,r_1]$  for some small interval $[r_0,r_1]$ in the $r$ direction. Finally, reconstruct $\widetilde{\Lambda}$ as the boundary of this surface by gluing in twisted bands to recover the resolved crossings.    Call the result $\widetilde{S}$.

\textbf{Step 3:} Finally, we map $\widetilde{M}$ to $M$, gluing segments of $\partial\widetilde{S}$ identified in the quotient.  To complete the construction of a Seifert surface for our original knot $\Lambda$, we remove the links components added in Step 2 by filling each one in with a disc  lying in the page.  
 This construction of a Seifert surface for $\Lambda_H$ renders the proof of Theorem~\ref{thm:tb} straightforward in the case that $\Lambda$ is nullhomologous in $\Sigma \times [0,1]$.

\begin{proof}[Proof of Theorem~\ref{thm:tb}, Part 1]
Here we address the case when $[\Lambda] = 0 \in H_1(\Sigma \times [0,1])$.
The Thurston-Bennequin number of $\Lambda$ is the linking number between $\Lambda$ and a vertical push-off $\Lambda'$ of $\Lambda$, or equivalently,   the signed intersection number of $\Lambda'$ and a Seifert surface for $\Lambda$.  By construction, this agrees with the intersection number between $\Lambda'$ and  the Seifert surface $\widetilde{S}$ for $\widetilde{\Lambda}$ in $\widetilde{M}$ constructed above.  

Recall that the surface $ \widetilde{S}$ was constructed from discs chosen to have a positive tangent component in the $\partial_r$ direction near $\Lambda$.  This implies immediately that for each pair of left\slash right cusps in $\mathcal{F}(\Lambda)$, there is a single negative intersection point between $\widetilde{S}$ and $\Lambda'$; this is easily seen  using the contactomorphism between standard pieces and a quotient of $(\mathbb{R}^3, \xi_{\text{std}})$.

Twisted bands also contribute intersection points between $\widetilde{S}$ and $\Lambda'$, and the sum of these contributions is the total writhe of the original diagram.  
Each crossing in the original front projection leads to an intersection point whose sign  is most easily computed via the following movie picture:

   \begin{figure}
\begin{center}
\scalebox{.7}{\includegraphics{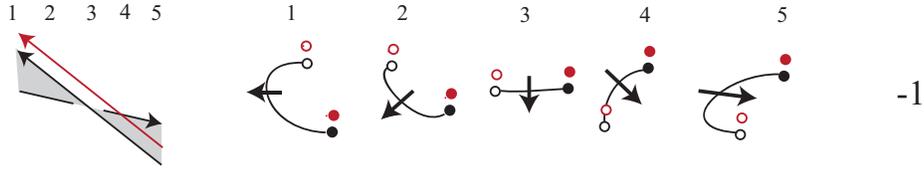}}
\caption{ The right-hand picture shows five $rt$-cross-sections of the twisted band shown on the left.  On each slice, the transverse arrow points in the positive direction of the surface, so that the sign of the intersection is $-1$. }\label{fig:bandmovie}
\end{center}
\end{figure}

Furthermore, each crossing between a curve of $\mathcal{F}(\Lambda)$ and a curve of $\mathcal{F}(\widetilde{\Lambda})\setminus \mathcal{F}(\Lambda)$ contributes an intersection point; note that the latter will always be the under-strand. Rotating the under-strand in Figure~\ref{fig:bandmovie} counterclockwise to vertical shows that the sign of the intersection is multiplied by $-m$.  

As noted above,  the simple closed curves added in Step 2 do not affect  the intersection number between $\Lambda'$ and $S$.

\end{proof}

\subsection{Seifert surfaces II}\label{sec:ss2}

We now turn our attention to the case of a Legendrian knot $\Lambda$ which is nullhomologous in $M$, but not in $\Sigma\times [0,1]$.  As described above, a Seifert surface for $\Lambda$ may be cut along $\Sigma_0\times N(B)$ to yield a new Seifert surface for a link $\Lambda \cup X$.   The components of the link come in three distinct types, which we have identified as Type 1 pairs, Type 2 components, and $\Lambda$ itself.  
\begin{remark}
These new boundary components are not Legendrian,  but one may nevertheless consider their image on the Morse diagram after flowing by $V$, and thus we continue to use  ``front projection" to describe the resulting curves on the Morse diagram.
\end{remark}

\begin{proof}[Proof of Theorem~\ref{thm:tb}, Part 2]  To conclude the proof, we consider the case  when $[\Lambda] \neq 0 \in H_1(\Sigma \times [0,1])$ but $[\Lambda] = 0 \in H_1(M)$.

Having replaced the original $\Lambda$ by a link which is nullhomologous in $\Sigma\times [0,1]$, we are free to carry out the construction of a Seifert surface, and hence, the computation of the Thurston-Bennequin number, exactly as in Section~\ref{sec:ss1}.  We note  that $\mathcal{F}(\Lambda')$ will be disjoint from any Type $1$ curves and will represent the ``under" strand in any intersections with Type 2 curves, so the introduction of $X$ affects the count of intersections only inasmuch as it changes the multiplicities along the trace curves.  
\end{proof}

\subsection{Proof of Lemmas~\ref{lem:writhe} and \ref{lem:intnumb}}

Recall that Lemma~\ref{lem:writhe} asserted the independence of the total writhe of $\mathcal{F}(\Lambda \cup X)$ from the choice of $X$, while Lemma \ref{lem:intnumb} claimed that the signed sum of teleporting endpoints is zero on each connected component of the trace curves.  In fact, both results will follow easily once the  multiplicities are interpreted topologically.

\begin{proof}[Proof of Lemma~\ref{lem:intnumb}] The multiplicities assigned to intervals of the trace curves were defined combinatorially, but their role in the construction of a Seifert surface makes the definition more transparent.  The construction described above introduces the oriented segments of $\widetilde{\Lambda}\setminus \Lambda$ to $\partial \widetilde{M}$, and the multiplicity of an interval is a signed count of how many segments lie on the part of $\partial \widetilde{M}$ whose neighborhood flows to the trace curve.  The sign comes from the fact that each added segment connects a pair of endpoints of $\Lambda$ which intersect $\partial \widetilde{M}$ with opposite signs and at different $t$-values, so the front projection of this oriented curve points in either the positive or negative $t$ direction.

Lemma~\ref{lem:intnumb} is equivalent to the statement that the endpoints with opposite signs may be joined by such arcs.  In particular, it is clear that this count depends only on $\widetilde{\Lambda}$ and not on its extension to $\widetilde\Lambda$.  Since $\Lambda$ is nullhomologous in $\Sigma\times[0,1]$, there exists some Seifert surface for $\Lambda$, and its pre-image in $\widetilde{\Lambda}$ defines an extension $\widetilde{\Lambda}$ which suffices to ensure the existence of the multiplicity function.
\end{proof}

\begin{proof}[Proof of Lemma~\ref{lem:writhe}] 

To show that the total writhe is independent of the choice of $X$ extending $\Lambda$, suppose that $X_1$ and $X_2$ are two such choices. We may modify the links, preserving intersection numbers between $\mathcal{F}(X_i)$ and the trace curves, so that $X_1-X_2$ consists of a pair of nullhomologous links contained, respectively, in $\Sigma\times (0,\epsilon]$ and  $\Sigma \times[1-\epsilon, 1)$.  We will prove this claim below; in the case where the $X_i$ consist only of Type 1 components, no modification is necessary. 

Assuming the claim, we note that each link contained in the product $\Sigma\times(a,b)$ has a Seifert surface which may also be assumed to lie in the product $\Sigma\times(a,b)$. Thus the signed count of intersections between this Seifert surface and $\partial \widetilde{M}$ is $0$ near the original $\Lambda$, and thus the multiplicities contributing to the total writhe of $\Lambda$ are unaffected.

We conclude by describing how to modify the Type 2 components of $X_1-X_2$ so that they are disjoint from $\Sigma_0$.  As a first step, isotope each Type 2 component across the binding so that it lies in a neighborhood of $\Sigma_0$ and intersects it twice.  Since the collection is nullhomologous, the signed intersection number with $\Sigma_0$ is $0$, so we may perform saddle resolutions to remove all the intersections with $\Sigma_0$ without altering the intersections between the front projection and the trace curves.  The resulting modified link satisfies the claim.

\end{proof}

\bibliographystyle{amsplain} 
\bibliography{unifronts}

\end{document}

%% file: header.tex
\usepackage[usenames]{color}
\usepackage{graphicx}
\usepackage{rlepsf,epstopdf}
\usepackage{hyperref}

\usepackage{pinlabel}

\newcommand{\rr}{\ensuremath{\mathbb{R}}}

\theoremstyle{plain}
\newtheorem{thm}{Theorem}[section]

\newtheorem{lem}[thm]{Lemma}
\newtheorem{lemma}[thm]{Lemma}
\newtheorem*{thm1.1}{Theorem 1.1}
\newtheorem*{thm1.4}{Theorem 1.4}
\newtheorem*{thm1.5}{Theorem 1.5}
\newtheorem*{thm1.6}{Theorem 1.6}

\newtheorem{example}[thm]{Example}
\newtheorem{prop}[thm]{Proposition}

\theoremstyle{definition}
\newtheorem{defn}[thm]{Definition}

\theoremstyle{remark}
\newtheorem{rem}[thm]{Remark}

\numberwithin{equation}{section}


%% file: UniFront.bbl
\providecommand{\bysame}{\leavevmode\hbox to3em{\hrulefill}\thinspace}
\providecommand{\MR}{\relax\ifhmode\unskip\space\fi MR }
\providecommand{\MRhref}[2]{%
  \href{http://www.ams.org/mathscinet-getitem?mr=#1}{#2}
}
\providecommand{\href}[2]{#2}
\begin{thebibliography}{10}

\bibitem{BG}
K.~Baker and J.~E. Grigsby, \emph{Grid diagrams and {L}egendrian lens space
  links}, J. Symplectic Geom. \textbf{7} (2009), no.~4, 415--448.

\bibitem{CY}
Kai Cieliekbak and Yakov Eliashberg, \emph{From {S}tein to {W}einstein and
  {B}ack: {S}ymplectic {G}eometry of {A}ffine {C}omplex {M}anifolds},
  Colloquium Publications, vol.~99, American Mathematicial Society, 2012.

\bibitem{Et2}
John~B. Etnyre, \emph{Legendrian and transversal knots}, Handbook of knot
  theory, Elsevier B. V., Amsterdam, 2005, pp.~105--185. \MR{2179261
  (2006j:57050)}

\bibitem{FM}
Benson Farb and Dan Margalit, \emph{A primer on mapping class groups},
  Princeton Mathematical Series, vol.~49, Princeton University Press,
  Princeton, NJ, 2012. \MR{2850125 (2012h:57032)}

\bibitem{GK}
D.~T. {Gay} and R.~{Kirby}, \emph{{Indefinite Morse 2-functions; broken
  fibrations and generalizations}}, ArXiv e-prints (2011), to appear in Geom.
  Topol.

\bibitem{Gi}
Emmanuel Giroux, \emph{G\'eom\'etrie de contact: de la dimension trois vers les
  dimensions sup\'erieures}, Proceedings of the {I}nternational {C}ongress of
  {M}athematicians, {V}ol. {II} ({B}eijing, 2002) (Beijing), Higher Ed. Press,
  2002, pp.~405--414. \MR{1957051 (2004c:53144)}

\bibitem{Lick}
W.~B.~R. Lickorish, \emph{A finite set of generators for the homeotopy group of
  a {$2$}-manifold}, Proc. Cambridge Philos. Soc. \textbf{60} (1964), 769--778.
  \MR{0171269 (30 \#1500)}

\bibitem{OS}
Burak Ozbagci and Andr{\'a}s~I. Stipsicz, \emph{Surgery on contact 3-manifolds
  and {S}tein surfaces}, Bolyai Society Mathematical Studies, vol.~13,
  Springer-Verlag, Berlin; J\'anos Bolyai Mathematical Society, Budapest, 2004.
  \MR{2114165 (2005k:53171)}

\bibitem{Pav}
Elena Pavelescu, \emph{Braiding knots in contact {$3$}-manifolds},
  \textbf{253} (2011), no.~2, 475--487.

\bibitem{TW}
W.~P. Thurston and H.~E. Winkelnkemper, \emph{On the existence of contact
  forms}, Proc. Amer. Math. Soc. \textbf{52} (1975), 345--347. \MR{0375366 (51
  \#11561)}

\bibitem{Weinstein}
Alan Weinstein, \emph{Contact surgery and symplectic handlebodies}, Hokkaido
  Math. J. \textbf{20} (1991), no.~2, 241--251. \MR{1114405 (92g:53028)}

\end{thebibliography}
